\documentclass[11pt]{article}
\usepackage{amsmath}
\usepackage{amsthm}
\usepackage{amsfonts}
\usepackage{graphicx}
\usepackage{latexsym}
\usepackage{url}
\usepackage{bbm}
\usepackage{mathtools}
\mathtoolsset{showonlyrefs}

 \usepackage{todonotes}

\newcommand{\R}{\mathbb{R}}

\newcommand{\E}{\mathbb{E}}
\newcommand{\PP}{\mathbb{P}}
\newcommand{\N}{\mathbb{N}}
\newcommand{\Z}{\mathbb{Z}}

\newcommand{\F}{\mathbb{F}}
\newcommand{\1}{\mathbbm{1}}
\renewcommand{\P}{\mathbb P}

\newcommand{\g}{\mathcal G}

\newcommand{\f}{\mathcal F}

\newcommand{\toop}{\stackrel{\PP}{\longrightarrow}}
\newcommand{\schw}{\stackrel{d}{\longrightarrow}}
\newcommand{\eqschw}{\stackrel{d}{=}}
\newcommand{\stab}{\stackrel{\mathcal{L}-s}{\longrightarrow}}

\newcommand{\toas}{\stackrel{\mbox{\tiny a.s.}}{\longrightarrow}}

\newcommand{\bee}{\begin{equation}}
\newcommand{\eee}{\end{equation}}
\newcommand{\beea}{\begin{array}}
\newcommand{\eeea}{\end{array}}

\renewcommand{\theequation}{\arabic{section}.\arabic{equation}}

\theoremstyle{plain}
\newtheorem{prop}{Proposition}[section]

\newtheorem{theo}[prop]{Theorem}
\newtheorem{lem}[prop]{Lemma}

\theoremstyle{definition}

\newtheorem{rem}[prop]{Remark}

\begin{document}

\title{Limit theorems for \\ stationary increments L\'evy driven moving averages}
\author{Andreas Basse-O'Connor\thanks{Department
of Mathematics, Aarhus University, 
E-mail: basse@math.au.dk.} \and 
Rapha\"el Lachi\`eze-Rey\thanks{Department
of Mathematics, University Paris Decartes, 
E-mail: raphael.lachieze-rey@parisdescartes.fr.} \and
Mark Podolskij\thanks{Department
of Mathematics, Aarhus University,
E-mail: mpodolskij@math.au.dk.}}

\maketitle

\begin{abstract}
In this paper we present some new limit theorems for power variation of $k$th order increments of  
stationary increments L\'evy driven moving averages. In this infill sampling setting, the asymptotic theory
gives very surprising results, which (partially) have no counterpart in the theory of discrete moving averages.
More specifically, we will show that the first order limit theorems and the mode of convergence 
strongly depend on the interplay between the given order of the increments,
the considered power $p>0$, the Blumenthal--Getoor index $\beta \in (0,2)$ of the driving pure jump L\'evy process $L$ and the behaviour of the kernel function
$g$ at $0$ determined by the power $\alpha $. 
First order asymptotic theory essentially comprise three cases: stable convergence towards a certain infinitely divisible distribution,
an ergodic type limit theorem and convergence in probability towards an integrated random process. We also prove the second order limit theorem 
connected to the ergodic type result. When the driving L\'evy process $L$ is a symmetric $\beta$-stable process we  obtain two different limits: 
a central limit theorem and convergence in distribution towards a  $(1-\alpha )\beta$-stable random variable.

\ \

\bigskip


\end{abstract}

\section{Introduction and main results} \label{Intro}
\setcounter{equation}{0}
\renewcommand{\theequation}{\thesection.\arabic{equation}}

In the recent years there has been an increasing interest in limit theory for power variations of stochastic processes. Power variation functionals 
and related statistics play a major role in analyzing the fine properties of the underlying model, in stochastic integration concepts and statistical
inference. In the last decade asymptotic theory for power variations of various classes of stochastic processes has been intensively 
investigated in the literature.
We refer e.g.\  to \cite{BGJPS,J,JP,PV} for limit theory for power variations of It\^o semimartingales, to \cite{BNCP09,BNCPW09,Coeu,gl89,nr09} 
for the asymptotic results in the framework of fractional Brownian motion and related processes, and to \cite{Rosenblatt, Rosenblatt-1} for investigations 
of power variation of the Rosenblatt process.
  
In this paper we study the power variation of \textit{stationary increments L\'evy driven moving averages}. More specifically, we consider 
an infinitely divisible process with  stationary increments  $(X_t)_{t\geq 0}$, defined on a  probability space $(\Omega, \mathcal F,   \mathbb P)$, 
given as
\begin{align} \label{def-of-X-43}
X_t= \int_{-\infty}^t \big\{g(t-s)-g_0(-s)\big\}\, dL_s,
\end{align} 
where $L=(L_t)_{t\in \R}$ is a symmetric L\'evy process on $\R$ with  $L_0=0$, that is, for all $u\in \R$, $(L_{t+u}-L_u)_{t\geq 0}$ is a L\'evy process indexed by $\R_+$ which distribution is invariant under multiplication with $-1$. Furthermore,  $g$ and $g_0$ are deterministic  functions from $\R$ into $\R$  vanishing on $(-\infty,0)$. 
In the further discussion
we will need the notion of \textit{Blumenthal--Getoor index} of $L$, which is defined via
\begin{align} \label{def-B-G}
\beta:=\inf\Big\{r\geq 0: \int_{-1}^1 |x|^r\,\nu(dx)<\infty\Big\}\in [0,2],
\end{align} 
where $\nu$ denotes the L\'evy measure of $L$. When $g_0=0$, process $X$ is a moving average,  and in this case $X$ is a stationary process.  If $g(s)=g_0(s)=s^\alpha_+$,
$X$ is a so called \text{fractional L\'evy process}. In particular, when $L$ is a $\beta$-stable L\'evy motion with $\beta \in (0,2)$,
$X$ is called a linear fractional stable motion and it is self-similar with index $H=\alpha+1/\beta$; see e.g.\ \cite{SamTaq}
(since in this case the stability index and the Blumenthal--Getoor index coincide, they are both denoted by $\beta$). 

Probabilistic analysis of stationary increments L\'evy driven moving averages such as semimartingale property, fine scale structure
and integration concepts, have been investigated in several papers. We refer to the work
of \cite{BasRosSM,BCI, bls12,  bm08,FK} among many others. 
However, only  few results on the power variations of such processes are presently available, exceptions to this are \cite[Theorem~5.1]{BCI} and   \cite[Theorem~2]{sg14};  see     Remark~\ref{rem3} for a closer discussion of a result from \cite[Theorem~5.1]{BCI}. These two results are concerned with certain power variation of fractional L\'evy process and have some overlad with our Theorem~\ref{maintheorem}(ii) for the linear fractional stable motion, but we apply different proofs.  The aim of this paper is to derive a rather complete picture of the first order asymptotic theory for power variation of the process $X$, and,
in some cases, the associated second order limit theory. We will see that the type of convergence and the limiting random variables/distributions
are quite surprising and novel in the literature. Apart from pure probabilistic interest, limit theory for power variations of 
stationary increments L\'evy driven moving averages give rise to a variety of statistical methods (e.g.\  identification and estimation of the
unknown parameters $\alpha$ and $\beta$) and it provides a first step towards asymptotic theory for power variation of stochastic processes,
which contain $X$ as a building block. In this context let us mention stochastic integrals with respect to $X$ and L\'evy semi-stationary processes,
which have been introduced in \cite{BNBV}.

To describe our main results we need to introduce some notation and a set of assumptions. In this work    
we consider the $k$th order increments $\Delta_{i,k}^{n} X$ of $X$, $k\in \N$, that are defined by 
\begin{align} \label{filter}
\Delta_{i,k}^{n} X:= \sum_{j=0}^k (-1)^j \binom{k}{j} X_{(i-j)/n}, \qquad i\geq k.
\end{align}
 For instance, we have that $\Delta_{i,1}^n X = X_{\frac{i}{n}}-X_{\frac{i-1}{n}}$ and $\Delta_{i,2}^n X = X_{\frac{i}{n}}-
2X_{\frac{i-1}{n}}+X_{\frac{i-2}{n}}$. Our main functional  is the power variation computed on the basis of $k$th order filters:
\begin{align} \label{vn}
V(p;k)_n := \sum_{i=k}^n |\Delta_{i,k}^{n} X|^p, \qquad p>0.
\end{align}
Now, we introduce the following set of assumptions on $g$ and $\nu$:
   
\noindent \textbf{Assumption~(A):}
The function $g\!:\R\to\R$ satisfies 
\begin{align}\label{kshs}
g(t)\sim c_0 t^\alpha\qquad \text{as } t\downarrow 0\quad \text{for some }\alpha>0\text{ and }c_0\neq  0, 
\end{align}
where $g(t)\sim f(t)$ as $t\downarrow 0$ means that $\lim_{t\downarrow 0}g(t)/f(t)= 1$. 
For some $\theta\in (0,2]$, $\limsup_{t\to \infty} \nu(x\!:|x|\geq t) t^{\theta}<\infty$ and $g - g_0$ is a bounded function in $L^{\theta}(\R_+)$.
 Furthermore,  $g$ is $k$-times continuous differentiable on $(0,\infty)$ and there  exists a $\delta>0$ such that  $|g^{(k)}(t)|\leq K t^{\alpha-k}$ for all $t\in (0,\delta)$,  $g^{(k)}\in L^\theta((\delta,\infty))$ and $|g^{(k)}|$ is decreasing on $(\delta,\infty)$.

\noindent \textbf{Assumption~(A-log):}
In addition to  (A) suppose that $\int_\delta^\infty |g^{(k)}(s)|^\theta \log(1/|g^{(k)}(s)|)\,ds<\infty$. 

Assumption~(A) ensures that the process $X$ is well-defined, cf.\ Section~\ref{sec2.4}. When $L$ is a $\beta$-stable L\'evy process, we always choose $\theta = \beta$ in assumption (A).
Before we introduce the main results, we need some more notation. 
Let $h_k\!:\R\to\R$ be given by 
\begin{align}\label{def-h-13}
h_k(x)&=  \sum_{j=0}^k (-1)^j \binom{k}{j} (x-j)_{+}^{\alpha},\qquad x\in \R,
\end{align} 
where  $y_+=\max\{y,0\}$ for all $y\in \R$. 
Let $\F=(\f_t)_{t\geq 0}$ be the  filtration generated by $(L_t)_{t\geq 0}$, $(T_m)_{m\geq 1}$  be a   sequence of $\F$-stopping times  that exhaust the jumps of $(L_t)_{t\geq 0}$, that is,  $ \{T_m(\omega):m\geq 1\}\cap \R_+ = \{t\geq 0: \Delta L_t(\omega)\neq 0\}$ and  $T_m(\omega)\neq T_n(\omega)$ for all $m\neq n$ with $T_m(\omega)<\infty$. Let     $(U_m)_{m\geq 1}$ be independent and uniform $[0,1]$-distributed random variables, defined on an extension $(\Omega', \mathcal F', \mathbb P')$ of the original probability space,
which are independent of $\f$. 

The following two theorems summarize the first and  second order limit theory for the power variation $V(p;k)_n$. We would like to emphasize 
part (i) of Theorem \ref{maintheorem} and part (i) of Theorem \ref{sec-order}, which are truly remarkable probabilistic results. We refer to \cite{Aldous,ren}
and to Section \ref{proofs-w3lkhj} for the definition of $\mathcal F$-stable convergence which will be denoted $\stab$.    
   
\begin{theo}[First order asymptotics.] \label{maintheorem}
Suppose (A) is satisfied and    assume that the Blumenthal--Getoor index satisfies $\beta<2$. 
We obtain the following three cases:
 
\begin{itemize}
\item[(i)] \label{mt-case-1} Suppose that (A-log) holds if  $\theta=1$. If $\alpha <k-1/p$ and $p>\beta$ we obtain the $\f$-stable convergence
\begin{equation} \label{part1}
n^{\alpha p}V(p;k)_n \stab |c_0|^p\!\!\!\!\! \!\sum_{m:\, T_m\in [0,1]} |\Delta L_{T_m}|^p V_m\quad\text{where}\quad 
V_m=\sum_{l=0}^{\infty} |h_k(l+U_m)|^p.
\end{equation} 
\item[(ii)] \label{mt-case-2} Suppose that  $L$ is a symmetric $\beta$-stable L\'evy process with scale parameter $\sigma>0$. If  $\alpha <k-1/\beta$ and   $p<\beta$ then it holds
\begin{align} \label{part2}
n^{-1+p(\alpha + 1/\beta)}V(p;k)_n \toop m_p
\end{align}
where  $m_p=|c_0|^p \sigma^p (\int_\R |h_k(x)|^\beta\,dx)^{p/\beta}\E[|Z|^p]$ and $Z$ is a symmetric $\beta$-stable random variable with scale parameter $1$. 

\item[(iii)] \label{mt-case-3} Suppose that $p\geq 1$. If  $p= \theta$ suppose in addition that (A-log) holds. For all    $\alpha>k-1/( \beta\vee p)$ 
we deduce
\begin{equation} \label{part3}
  n^{-1+pk}V(p;k)_n \toop \int_0^1 |F_u|^p\, du 
 \end{equation}
where $(F_u)_{u\in \R}$ is a measurable process satisfying
\begin{equation}
F_u=  \int_{-\infty}^u g^{(k)}(u-s) \,dL_s\quad \text{a.s.\ for all }u\in \R \quad \text{and}\quad \int_0^1 |F_u|^{p}\,du<\infty\quad \text{a.s.}
\end{equation} 
\end{itemize}
\end{theo}

We remark that,  
except the critical cases where $p= \beta$, $\alpha= k-1/p$ and $\alpha= k-1/\beta$, 
Theorem~\ref{maintheorem} covers all possible choices of $\alpha>0,\beta\in [0,2)$ and  $p\geq 1$. We also note that the limiting random variable in 
\eqref{part1} is infinitely divisible, see Section~\ref{sec2.4} for more details. In addition, we note that there is no convergence in probability in \eqref{part1} 
due to the fact that the random variables $V_m$, $m\geq 1$, are independent of $L$ and the properties of stable convergence. 
To be used in the next theorem we recall that a   totally right skewed $\rho$-stable random variable $S$ with $\rho>1$, mean zero and scale parameter $\eta>0$ has characteristic function given by 
\begin{equation}
\E[ e^{i \theta S} ] =\exp\Big( -\eta^\rho |\theta |^\rho \big( 1- i \mathrm{sign}(\theta)\tan(\pi \rho/2)\big)\Big),\qquad \theta\in \R. 
\end{equation}

For part (ii) of Theorem \ref{maintheorem}, which we will refer to as the ergodic case, we also show the second order asymptotic results. 
\begin{theo}[Second order assymptotics]\label{sec-order}
Suppose that assumption (A) is satisfied and  $L$ is a symmetric $\beta$-stable L\'evy process with scale parameter $\sigma>0$. 
Let $f:[0,\infty )\mapsto \R$ be given by  $f(t)=g(t)/t^{\alpha}$ for  $t>0$ and $f(0)= c_0$, and assume that $f$ is $k$-times  continuous right differentiable at $0$. For the below case (i) assume,  in addition, that $| g'(t)|\leq K t^{\alpha-1}$ for all $t>0$. 
\begin{itemize}
\item [(i)] If  $k=1$, $\alpha <1-1/\beta$ and $p<\beta/2$,  then it  holds that
\[
n^{1-\frac{1}{(1-\alpha)\beta}}\Big(n^{-1+p(\alpha + 1/\beta)}V(p;k)_n- m_p\Big) \schw S,
\]
where $S$ is a totally right skewed  $(1-\alpha)\beta$-stable random variable with mean zero and scale parameter  $\widetilde{\sigma}$, which is defined in  Remark~\ref{rem-const}(i).   
\item[(ii)]  \label{mt-case-6} If   $k\geq 2$, $\alpha <k-2/\beta$ and $p<\beta/2$ we deduce that
\begin{align} \label{part5}
\sqrt{n} \Big( n^{-1+p(\alpha + 1/\beta)}V(p;k)_n - m_p \Big) \schw \mathcal N(0, \eta^2),
\end{align} 
where the quantity $\eta^2$ is defined in Remark~\ref{rem-const}(ii).
\end{itemize}
\end{theo}

This paper is structured as follows. Section~\ref{sec2} presents some remarks about the nature and applicability of the main results, and it also
includes a discussion about related problems.  Section~\ref{secPrel} introduces some preliminaries.  
We state the proof of Theorem
\ref{maintheorem} in Section~\ref{proofs-w3lkhj}, while the proof of Theorem \ref{sec-order} is demonstrated in Section \ref{sec5}.
Some technical lemmas are deferred to the Appendix.

\section{Related results, remarks and extensions} \label{sec2}
\setcounter{equation}{0}
\renewcommand{\theequation}{\thesection.\arabic{equation}}
In this section we will give a review of related asymptotic results in the literature, present some intuition behind our limit theory
and discuss possible extensions.

\subsection{Fine properties of stationary increments L\'evy driven moving averages} \label{sec2.1}
In this subsection we will discuss the probabilistic properties of the process $X$ defined at \eqref{def-of-X-43}, such as
semimartingale property and small scale behaviour, and their consequences for the limit theory. For simplicity suppose that $L$ is a symmetric $\beta$-stable L\'evy process. 

Suppose, in addition, that $g'$ satisfies the lower bound $K t^{\alpha-1}\leq |g'(t)|$ for all $t\in (0,1)$ and a constant   $K>0$. By \cite[Example~4.9]{BasRosSM} it follows that  $X$ is a semimartingale if and only if $\alpha>1-1/\beta$, which is exactly condition (iii) in Theorem~\ref{maintheorem} when $k=1$ and $p<\beta$.

%

To better understand the limit theory stated in Theorem \ref{maintheorem}(ii) and Theorem \ref{sec-order}, which both refer to the ergodic case, we need
to study the small scale behaviour of the stationary increments L\'evy driven moving averages $X$. We intuitively deduce the 
following approximation for the increments of $X$ for a small $\Delta>0$:
\begin{align*}
X_{t+\Delta} - X_t &= \int_{\R} \{g(t+\Delta -s) - g(t -s)\}\, dL_s \\
& \approx \int_{t+\Delta -\epsilon }^{t+\Delta} \{g(t+\Delta -s) - g(t -s)\} \,dL_s \\
& \approx c_0 \int_{t+\Delta -\epsilon }^{t+\Delta} \{(t+\Delta -s)_{+}^\alpha  - (t -s)_{+}^\alpha \} \,dL_s \\
& \approx c_0 \int_{\R} \{(t+\Delta -s)_{+}^\alpha  - (t -s)_{+}^\alpha \} \,dL_s = \widetilde{X}_{t+\Delta} - \widetilde{X}_t,
\end{align*} 
where 
\begin{align} \label{flm1}
\widetilde{X}_t := c_0 \int_{\R} \{(t-s)_{+}^\alpha  - ( -s)_{+}^\alpha \}\, dL_s,
\end{align} 
and $\epsilon >0$ is an arbitrary small real number with $\epsilon \gg \Delta$. In the classical terminology $\widetilde{X}$ is called
the \textit{tangent process of $X$}.   
The formal proof of this first order approximation, which will be 
demonstrated in Section~\ref{proofs-w3lkhj}, relies 
on assumption (A) and the fact that, under conditions of Theorem \ref{maintheorem}(ii), the weight $g(t+\Delta -s) - g(t -s)$ attains asymptotically 
highest values when $s\approx t$, since $g'$ explodes at $0$. Recall that the process $\widetilde{X}$ is the linear  fractional stable motion.
In particular, under the assumptions of Theorem \ref{maintheorem}(ii), it is a linear fractional stable motion with $\beta$-stable marginals 
and self-similarity index $H=\alpha +1/\beta$.   
Thus, one may transfer the first order asymptotic 
theory for power variation of $\widetilde{X}$ to the corresponding results for power variation of $X$. However, the law of large numbers
for power variation of $\widetilde{X}$ is easier to handle than the original statistic due to self-similarity property of $\widetilde X$, which allows to transform the original triangular observation scheme
into a usual one when studying distributional properties. Then, the standard ergodic limit theory becomes applicable. Indeed, this is exactly
the method of proof of Theorem \ref{maintheorem}(ii). We remark however that it is much more technical to use the relationship between $X$ and $\widetilde{X}$
for the proof of the second order asymptotic theory in Theorem \ref{sec-order}. In fact, we use a more direct approach to show the results of Theorem \ref{sec-order}.

\subsection{Limit theory in the Gaussian case} \label{sec2.2}
Throughout this subsection we recall the asymptotic theory for power variation of fractional Brownian motion $(B_t^H)_{t\geq 0}$ with Hurst
parameter $H\in (0,1)$ and relate it to our limit theory. The main demonstrated results have been established in the classical work \cite{BM83,T79}.
We write $V(B^H,p;k)_n$ to denote the power variation statistics defined at \eqref{vn} associated with the fractional Brownian motion $(B_t^H)_{t\geq 0}$. 

First of all, we observe the law of large numbers
\begin{align*}
n^{-1+pH}V(B^H,p;k)_n \toop m_p:= \E[|n^H\Delta_{i,k}^{n} B^H|^p], 
\end{align*} 
which follows from the ergodic theorem (note that $m_p$ is independent of $n$ due to self-similarity property of $B^H$). 
The associated weak limit theory depends on the interplay between the correlation kernel of the fractional Brownian noise 
and the \textit{Hermite rank} of the function $h(x)=|x|^p - m_p$. Recall that the correlation kernel
$\rho_k(j)=\text{corr}(n^H\Delta_{k,k}^{n} B^H, n^H\Delta_{k+j,k}^{n} B^H)$ of $k$th order differences
of the fractional Brownian motion satisfies that 
\[
|\rho_k(j)|\leq K j^{2H-2k} \qquad \text{for }j\geq 1,
\]
for some $K>0$. The Hermite expansion of the function $h$ is defined as
\[
h(x)=|x|^p - m_p = \sum_{l=2}^{\infty} \lambda_l H_l(x),
\]
where $(H_l)_{l\geq 0}$ are Hermite polynomials, i.e.\ 
\[
H_0(x)=1 \quad \text{and} \quad H_l (x)= (-1)^l \exp(x^2/2) \frac{d^l}{dx^l} \{-\exp(x^2/2)\} \quad \text{for } l\geq 1.
\]
The Hermite rank of $h$ is the smallest index $l$ with $\lambda_l \not =0$, which is $2$ in our case. The condition for the validity
of a central limit theorem associated to a standardized version of $V(B^H,p;k)_n$ is then
\[
\sum_{j=1}^\infty \rho_k^2 (j)<\infty, 
\] 
where the power $2$ indicates the Hermite rank of $h$. The latter is obviously fulfilled for any $k\geq 2$ and also for $k=1$ if $H\in (0,3/4)$. 
The next result is a famous statement from \cite{BM83,T79}. The Gaussian limit case is usually referred to as Breuer--Major central limit theorem. 

\begin{theo} \label{mb-th}
The following assertions hold:
\begin{itemize}
\item[(i)] Assume that $k\geq 2$ or $k=1$ and $H\in (0,3/4)$. Then  the central limit theorem holds
\begin{align*}
\sqrt{n} \left( n^{-1+pH}V(B^H,p;k)_n - m_p \right) \schw \mathcal N(0, v_p), 
\end{align*}
where $v_p= \sum_{l=2}^\infty l! \lambda_l^2 \Big(1+ 2 \sum_{j=1}^\infty \rho_k^l(j) \Big)$. 
\item[(ii)] When $k=1$ and $H=3/4$ we have
\begin{align*}
\frac{\sqrt{n}}{\log n} \left( n^{-1+pH}V(B^H,p;k)_n - m_p \right) \schw \mathcal N(0, \tilde{v}_p), 
\end{align*}
where $\tilde{v}_p = 4 \lambda_2 \lim_{n\rightarrow \infty} \frac{1}{\log n} \sum_{j=1}^{n-1} \frac{n-k}{n} \rho_1^2(j)$.
\item[(iii)] When $k=1$ and $H\in (3/4,1)$ it holds that 
\begin{align*}
n^{2-2H} \left( n^{-1+pH}V(B^H,p;k)_n - m_p \right) \schw Z, 
\end{align*}
where $Z$ is a Rosenblatt random variable.
\end{itemize}
\end{theo}
The results of Theorem \ref{mb-th} has been extended to the case of general stationary increments Gaussian processes in \cite{gl89}. Asymptotic 
theory for power variation of stochastic integrals with respect to Gaussian processes has been intensively studied in \cite{BNCP09,BNCPW09,cnw06}. 
We also refer to the interesting work  \cite{Rosenblatt-1} for a study of quadratic variation of the Rosenblatt process.

Summarizing the asymptotic theory in the Gaussian case, we can conclude that the limiting behaviour in the framework of stationary increments L\'evy driven moving averages is quite different. Not surprisingly, 
the quite stunning results of Theorem \ref{maintheorem}(i) and Theorem \ref{sec-order}(i) do not appear in the Gaussian setting
(the convergence of the type \eqref{part3} may very well appear for differentiable Gaussian processes).

\subsection{Limit theorems for discrete moving averages} \label{sec2.3}
Asymptotic theory for statistics of discrete moving averages has been a subject of a deep investigation during the last thirty years. Indeed,
the variety of different limiting distributions, which may appear under certain conditions on the innovations and weight coefficients, is quite astonishing.
In a functional framework
they include Brownian motion, $m$th order Hermite  processes, stable L\'evy processes with various stability indexes and fractional Brownian motion.
We refer to the work \cite{at87,hh97,h99,ks01,s02,s04} among many others. The limit theory is much more diverse and still not completely understood
in contrast to the Gaussian case discussed in the previous subsection. For this reason, we will rather concentrate on some asymptotic results 
related to our set of conditions.

Let us consider a discrete moving average $(Z_i)_{i\in \mathbb Z}$ of the form
\begin{align*}
Z_i= \sum_{j=1}^{\infty} b_j \zeta_{i-j},
\end{align*} 
where $(\zeta_i)_{i\in \mathbb Z}$ is a sequence of i.i.d.\ random variables with mean $0$ and variance $1$, and $(b_j)_{j\geq 1}$ are non-random coefficients. The innovation $\zeta_1$ is assumed to satisfy:  There exists $K,\delta>0$ such that for all  $u\in \R$, $|\E[ e^{ i u \zeta_1}]|\leq K(1+|u|)^{-\delta}$.
For simplicity of exposition and comparison we assume that the distribution of $\zeta_1$ is symmetric. Now, we briefly review the results of \cite{s04}. 
The assumptions on the decay of the coefficients $(b_j)_{j\geq 1}$ and the tail behaviour of the innovations are as follows:  
\begin{align*}
b_j \sim k_0 j^{-\gamma} \quad \text{as } j\rightarrow \infty, \qquad \mathbb P(|\zeta_1| >x) \sim q x^{-\beta} \quad \text{as } x\rightarrow \infty,
\end{align*}   
for some $\gamma \in (1/2,1)$, $\beta \in (2,4)$ and $k_0,q\neq 0$. Surgailis \cite{s04} studies the asymptotic behaviour of the statistic
\begin{align*}
S_n = \sum_{i=1}^n h(Z_i),
\end{align*} 
where $h:\R\to\R$ is a \textit{bounded} measurable function with $\mathbb E[h(Z_1)]=0$. In this framework the most important ingredient is the \textit{Appell rank}
of the function $h$ (cf.\ \cite{at87}). It is defined as
\begin{align*}
k^{\star} := \min_{k\geq 1} \{h_\infty^{(k)} (0) \not= 0\} \qquad \text{with} \qquad h_\infty (x):= \mathbb E[h(Z_1+x)].  
\end{align*} 
The Appell rank is similar in spirit with the Hermite rank introduced in the previous section, but it is much harder
to prove limit theorems for the statistic $S_n$ for an arbitrary Appell rank $k^{\star}$. The main problem is that, in contrast to Hermite expansion, the
expansion with respect to Appell polynomials typically does not exist. Furthermore, the decomposition of $S_n$ becomes more complex when $k^{\star}$
increases. For this reason only the cases $k^{\star}=1,2,3$ are usually treated in the literature in the framework of heavy
tailed innovations. In particular, \cite{s04} investigates the cases
$k^{\star}=2,3$.

At this stage we compare the introduced setting of discrete moving average  with our framework of \eqref{def-of-X-43}. For the sake 
of exposition, we will rather consider the tangent process $\widetilde{X}$ defined at \eqref{flm1} driven by a symmetric $\beta$-stable
L\'evy motion $L$. We immediately see that our assumption on $\beta$, namely $\beta \in (0,2)$, does not satisfy the tail behaviour condition 
on the innovations introduced above (in particular, $\E[L_t^2]=\infty$). As for the weight coefficients, our kernel function satisfies that 
\[
c_0 ((x+1)^{\alpha} - x^\alpha) \sim \alpha c_0 x^{\alpha -1} \quad \text{as } x\rightarrow \infty. 
\]
Thus, the connection to the setting of discrete moving averages is given via $k_0=\alpha c_0$ and $\gamma = 1-\alpha$ (indeed, it holds 
that $1-\alpha\in (1/2,1)$ under conditions of Theorem \ref{maintheorem}(ii)). In our framework, the function $h$ is given via $h(x)=|x|^p - m_p$,
where the quantity $m_p$ has been defined in Theorem \ref{maintheorem}(ii), which is obviously not bounded. Since $h$ is an even function and $L$ is symmetric,
we readily deduce that $k^{\star}=2$. 

Now, we summarize the asymptotic theory from \cite{s04} for the statistic $S_n$ in the case of Appell rank $k^{\star}=2$ ($2<\beta <8/3$):
\begin{itemize}
\item[(i)] $1/2<\gamma <(\beta + \sqrt{\beta ^2-2\beta})/2\beta$: convergence rate $n^{2-2\gamma}$, Rosenblatt limit.
\item[(ii)] $(\beta + \sqrt{\beta ^2-2\beta})/2\beta<\gamma <2/\beta $: convergence rate $n^{1/\gamma \beta}$,  $\gamma \beta$-stable limit.
\item[(iii)] $2/\beta < \gamma < 1$: convergence rate $n^{1/2}$, normal limit.
\end{itemize}
Although the results of \cite{s04} are not directly applicable (recall that in our setting $\beta \in (0,2)$ and $h$ is unbounded), Theorem 
\ref{sec-order}(i) corresponds to case (ii) of \cite{s04}. Indeed, we apply a similar proof strategy to show the weak convergence. However, 
strong modifications due to unboundedness of $h$, triangular nature of summands in \eqref{vn}, stochastic integrals instead of sums,  and the different set of conditions are required. 

\begin{rem} \label{rem1} \rm
Case (ii) of \cite{s04} is a quite remarkable result, since a  $\gamma \beta$-stable distribution appears in the limit 
although the summands of $S_n$ are bounded (in particular, all moments of $S_n$ exist). In particular, the rate of convergence 
$n^{1/\gamma \beta}$ does not correspond to the variance of $S_n$. \qed 
\end{rem}

\begin{rem} \label{rem2} \rm
The symmetry condition on the L\'evy process $L$ is assumed for sake of assumption simplification. Most asymptotic results of this paper
would not change if we dropped this condition. However, the Appell rank of the function $h(x)=|x|^p - m_p$ might be $1$ when $L$ is not symmetric
and this does change the result of Theorem \ref{sec-order}(i). More specifically, the limiting distribution turns out to be $\beta$-stable 
(see e.g.\  \cite{ks01} for the discrete case). We dispense with the exact exposition of this case. \qed  
\end{rem}

\subsection{Further remarks and possible extensions} \label{sec2.4}
We start by commenting on the set of conditions introduced in assumption (A). First of all, 
it follows by \cite[Theorem~7]{RajRos} that the process $X$, defined in \eqref{def-of-X-43},  is well-defined if and only if for all $t\geq 0$, 
\begin{equation}\label{sdkfhsd;kf}
\int_{-t}^\infty \int_\R \Big(\big| f_t(s)x\big|^2\wedge 1\Big)\,\nu(dx)\,ds<\infty,
\end{equation}
where $f_t(s)=g(t+s)-g_0(s)$.
By adding  and subtracting  $g$ to $f_t$ it follows by assumption (A) and the  mean value theorem  that  $f_t\in L^\theta(\R_+)$ and $f_t$ is  bounded. 
For all $\epsilon>0$, assumption (A) implies that 
\begin{equation}\label{eq:74}
\int_{\R} (|y x|^2\wedge 1) \,\nu(dx)\leq K\Big( \1_{\{|y|\leq 1\}} |y|^\theta +\1_{\{|y|>1\}}|y|^{\beta+\epsilon}\Big), 
\end{equation}
which shows  \eqref{sdkfhsd;kf} since $f_t\in L^\theta(\R_+)$  is bounded.  
We  remark that for $\theta'<\theta$ it holds
\[
\limsup_{t\to \infty} \nu(x\!:|x|\geq t) t^{\theta}<\infty \quad \Longrightarrow \quad \limsup_{t\to \infty} \nu(x\!:|x|\geq t) t^{\theta'}<\infty.
\]
On the other hand, the assumption $g^{(k)}\in L^{\theta '} ((\delta ,\infty ))$ is stronger than $g^{(k)}\in L^{\theta } ((\delta ,\infty ))$, which creates
a certain balance between these two conditions. Finally, we note that the assumption (A-$\log$) will be used only for the case $\theta =1$ (resp.\ $\theta=p$) in part (i) (resp.\ part (iii)) of 
Theorem \ref{maintheorem}.

More importantly, the conditions of assumption (A) guarantee that the quantity
\[
\int_{-\infty}^{t-\epsilon} g^{(k)}(t-s) \,dL_s, \qquad \epsilon >0,
\]               
is well-defined (typically, the above integral is not well-defined for $\epsilon =0$). The latter is crucial for the proof of Theorem \ref{maintheorem}(i).
We recall that the condition $p\geq 1$ is imposed in Theorem \ref{maintheorem}(iii). 
We think that this condition might not be necessary, but the results
of \cite{SamBra} applied in our proofs require $p\geq 1$. However, when the index $\alpha$ further
satisfies $\alpha>k$, then the stochastic process $F$ at \eqref{part3} is continuous and condition
$p\geq 1$ is not needed in Theorem \ref{maintheorem}(iii).

The conditions $\alpha \in (0,k-1/p)$ and $p>\beta$ of Theorem \ref{maintheorem}(i) seem to be sharp.
Indeed, since $|h_k(x)|\leq Kx^{\alpha -k}$ for large $x$, 
we obtain from \eqref{part1} that 
\[
\sup_{m\geq 1} V_m <\infty  
\]
when $\alpha \in (0,k-1/p)$. On the other hand $\sum_{m:T_m\in [0,1]} |\Delta L_{T_m}|^p<\infty$ for $p>\beta$, which follows from the definition
of the Blumenthal--Getoor index at \eqref{def-B-G}. Notice that under assumption $\alpha \in (0, k-1/2)$
the case $p=2$, which corresponds to quadratic variation, always falls under Theorem \ref{maintheorem}(i).  
We remark that the distribution of the limiting variable in \eqref{part1} does not depend
on the chosen sequence $(T_m)_{m\geq 1}$ of stopping times which exhausts the jump times of $L$. Furthermore, 
the limiting random variable $Z$ in \eqref{part1} is infinitely divisible with L\'evy 
measure $(\nu\otimes \eta)\circ \big((y,v)\mapsto  |c_0 y|^p v\big)^{-1}$, where $\eta$ denotes the law of $V_1$. In fact, $Z$ has characteristic function
given by 
\begin{equation}
\E[\exp(i \theta Z)]=\exp\Big( \int_{\R_0\times \R} (e^{i \theta |c_0 y|^p v}-1)  \,\nu(dy)\,\eta(dv)   \Big).
\end{equation}
To show this, let  $\Lambda$ be the Poisson random measure given by 
$\Lambda=\sum_{m=1}^\infty \delta_{(T_m,\Delta L_{T_m})}$ on $[0,1]\times \R_0$  which has intensity measure $\lambda\otimes \nu$. Here  $\R_0:=\R\setminus \{0\}$ and  $\lambda$ denotes the Lebesgue measure on $[0,1]$. 
Set  $\Theta=\sum_{m=1}^\infty \delta_{(T_m, \Delta L_{T_m},V_m)}$. Then $\Theta$ is a Poisson random measure with intensity  measure $\lambda\otimes \nu \otimes \eta$, due to  \cite[Theorem~36]{Serfozo}, and hence the  above claims  follows from the stochastic integral representation
\begin{equation}
Z= \int_{[0,1]\times \R_0\times \R} \big(|c_0 y|^p v\big) \, \Theta(ds,dy,dv).
\end{equation} 
As for Theorem \ref{maintheorem}(iii), we remark that for values of $\alpha$ close to $k-1/p$ or $k-1/\beta $, 
the function $g^{(k)}$ explodes at $0$. This leads to unboundedness of the process $F$ defined in 
Theorem \ref{maintheorem}(iii). Nevertheless,  the limiting random variable in \eqref{part3}
is still finite. 

We recall that $L$ is assumed to be a symmetric $\beta$-stable L\'evy process   in Theorems \ref{maintheorem}(ii) and \ref{sec-order}. This assumption
can be relaxed following the discussion of tangent processes in Section \ref{sec2.1}. Indeed, only the small scale behaviour of the driving L\'evy process $L$
should matter for the statement of the aforementioned results. When the small jumps of $L$ are in the domain of attraction of a symmetric $\beta$-stable L\'evy
process, e.g.\   its L\'evy measure satisfies the decomposition
\[
\nu (dx) = \left(\text{const} \cdot |x|^{-1-\beta} + \varphi(x) \right) dx
\]         
with $\varphi(x)=o(|x|^{-1-\beta})$ for $x\rightarrow 0$, the statements of Theorems \ref{maintheorem}(ii) and \ref{sec-order} should remain valid under
possibly further assumptions on the function $\varphi$. Such processes include for instance  tempered or truncated symmetric $\beta$-stable L\'evy processes. 

\begin{rem} \label{rem3} \rm
Theorem~5.1 of \cite{BCI} studies the first order asymptotic of the power variation  of some fractional fields $(X_t)_{t\in \R^d}$. In the  case  $d=1$, they considers  fractional L\'evy processes $(X_t)_{t\in \R}$ of the form 
\begin{equation}\label{eq:7237}
 X_t= \int_{\R} \big\{|t-s|^{H-1/2} - |s|^{H-1/2}\big\}\,dL_s
\end{equation} where $L$ is a truncated $\beta$-stable L\'evy process. This  setting is close to fit into the framework of the present paper \eqref{def-of-X-43} with $\alpha= H-1/2$ except for the fact that  the stochastic integral \eqref{eq:7237} is over the hole real line. However, the proof of Theorem~\ref{maintheorem}(i) still holds for $X$ in \eqref{eq:7237} with  obvious modifications of $h_k$ and $V_m$ in \eqref{def-h-13} and \eqref{part1}, respectively.  
Notice also that   \cite{BCI} considers  the power variation along the subsequence $2^n$, which corresponds to dyadic partitions, and their setting includes second order increments.  For $p<\beta$,  Theorem~5.1 of \cite{BCI} claims that 
$2^{n \alpha p} V(p;2)_{2^n}\to C$ a.s.\ where $C$ is a positive constant, which in the notation of \cite{BCI}   corresponds to the case $\alpha<\beta<2$. However, this contradicts Theorem~\ref{maintheorem}(i) together with the remark following it, namely that,   convergence in probability can not take place under the conditions of  Theorem~\ref{maintheorem}(i) not even trough a subsequence. The main part of the proof of the cited result, \cite[Theorem~5.1]{BCI}, consists in proving  that 
$\E[  2^{n\alpha p} V(p;2)_{2^n}] \to C$ (see p.~372, l.~11), and this result agrees with our Theorem~\ref{maintheorem}(i).  However, in the last three lines of their proof (p.\ 372, l.\ 12--15) it is argued how $\E[  2^{n\alpha p} V(p;2)_{2^n}] \to C$ implies that $2^{n\alpha p} V(p;2)_{2^n}\to C$ a.s.\ The argument relies on the statement  from \cite{BCI}: ``$S_n(y)= \E[ S_n(y) ] (1+o_{(a.s.)}(1))$'', where the stochastic process $S_n(y)$ is the empirical characteristic function of normalised increments
defined on \cite[p.~367]{BCI}. This statement   is only shown for each fixed $y$ (see \cite[Lemma~5.2]{BCI}), but to use it to carry out the proof, it is crucial to know the dependence of 
the $o_{(a.s.)}(1)$-term in $y$. In  fact,  it is not even enough to have  boundedness in $y$ of the  $o_{(a.s.)}(1)$-term. 
\qed
\end{rem}

Let us further explain the various conditions of Theorems \ref{maintheorem}(ii) and \ref{sec-order}. The condition $p<\beta$ obviously ensures the existence
of moments $m_p$, while assumption $p<\beta/2$ ensures the existence of variance of the statistic $V(p;k)_n$.
The validity range of the central limit theorem ($\alpha \in (0,k-2/\beta)$) in \eqref{part5} is smaller than the validity range of the law of large numbers
in Theorem \ref{maintheorem}(ii) ($\alpha \in (0,k-1/\beta)$). It is not clear which limit distribution appears in case of $\alpha \in (k-2/\beta, k-1/\beta)$.
There are also two critical cases that correspond to $\alpha = k - 1/p$ in Theorem \ref{maintheorem}(i) and $\alpha = k - 1/ \beta$ in Theorem 
\ref{maintheorem}(ii).
We think that additional logarithmic rates will appear in these cases, but the precise proofs are a subject of future research.

\begin{rem}\label{rem-const}
\begin{enumerate}
\item [(i)] To define the constant $\tilde \sigma$  appearing in Theorem~\ref{sec-order}(i)
we set 
\begin{align}
\kappa = {}&\frac{\alpha^{1/(1-\alpha)} }{1-\alpha} 
\int_0^{\infty} \Phi (y) y^{-1- 1/(1-\alpha )} \,dy,
\end{align}
where   $\Phi(y):= \E[|\widetilde{X}_1+y|^p - |\widetilde{X}_1|^p]$, $y\in\R$, and   $\widetilde{X}_t$ is a linear fractional stable motion  defined in \eqref{flm1} with $c_0=1$ and $L$ being a standard symmetric $\beta$-stable L\'evy process.
In addition, set 
\begin{equation}\label{def-tau-rho}
 \tau_{\rho}=  \frac{\rho-1}{\Gamma(2-\rho)|\cos(\pi \rho/2)|},\qquad \text{for all }\rho\in (1,2),
\end{equation}
where $\Gamma$ denotes the gamma function. Then, 
\begin{equation}
\tilde \sigma = |c_0|^p \sigma^p  \Big(\frac{\tau_\beta}{\tau_{(1-\alpha)\beta}}\Big)^{\frac{1}{(1-\alpha)\beta}}\kappa .
\end{equation}
The function $\Phi(y)$  
can be computed explicitly, see \eqref{rep-H-est-2}.  
This representation shows, in particular, that $\Phi (y)>0$ for all $y>0$, and hence the limiting variable 
$S$
in Theorem \ref{sec-order}(i) is not degenerate, because $\tilde \sigma >0$.
\item [(ii)] The constant $\eta^2$ in Theorem~\ref{sec-order}(ii)  is given by 
\begin{equation}\label{def-eta-23}
\eta^2 =  |c_0 \sigma |^{2p}\Big(\theta(0) + 2 \sum_{i=1}^\infty \theta(i)\Big)  , \quad
\theta(i) = a_p^{-2} \int_{\R^2} \frac{1}{|s_1s_2|^{1+p}} \psi_i(s_1,s_2)\, ds_1 \,ds_2,
\end{equation}
\begin{align}
\label{theta} \psi_i(s_1,s_2) = {}&  \exp \left(-  \int_{\R} |s_1 h_k(x) -s_2h_k(x+i)|^{\beta } \,dx \right)  
  \\[1.5 ex]
{}&- \exp \left(-  \int_{\R} |s_1 h_k(x)|^{\beta} + |s_2h_k(x+i)|^{\beta} \,dx \right), \nonumber
\end{align}
where the function $h_k$ is defined at \eqref{def-h-13} and $a_p:= \int_{\R} (1-\exp(iu)) |u|^{-1-p} \,du$.
\end{enumerate}
\qed
\end{rem}

\begin{rem} \label{rem4} \rm
Let us explain the somewhat complex form of the variance $\eta^2$ in  \eqref{def-eta-23}. A major problem of proving Theorems \ref{sec-order}(ii) 
is that the covariance structure of stationary summands of $V(p;k)_n$ can not be computed  directly. 
However, the identity
\begin{align} \label{xp}
|x|^p = a_p^{-1} \int_{\R} (1-\exp(iux)) |u|^{-1-p} du \qquad \text{for } p\in (0,1),
\end{align}
which can be shown by substitution $y=ux$ ($a_p$ is defined in Remark~\ref{rem-const}(ii)), turns out to be a useful instrument. 
Indeed, for any deterministic function $\varphi :\R\rightarrow \R$
satisfying $\varphi\in L^{\beta}(\R)$, it holds that
\begin{align} \label{charfun}
\E\left[ \exp \left(iu \int_{\R} \varphi(s) \,dL_s \right)\right] = \exp \left(-\sigma^{\beta}|u|^{\beta} \int_{\R} |\varphi(s)|^{\beta} \,ds \right),
\end{align}
where $\sigma>0$ is the scale parameter of $L$. These two identities are used to compute the variance of the statistic $V(p;k)_n$ and they 
are both reflected in the formula for the quantity $\theta(i)$ in  \eqref{def-eta-23}. \qed
\end{rem}

\begin{rem} \label{rem6} \rm
A key to the proof of Theorem \ref{sec-order}(ii) is a quite technical Lemma \ref{helplem2}, 
which gives an upper bound for the covariances  of the summands of $V(p;k)_n$. We do believe that the obtained decay rates, which are explicitly derived in 
the proofs of Lemma \ref{helpLemma} and \ref{helplem2},  are essentially sharp (possibly up to a log rate) and they might be of independent interest. 
Our estimation method is based upon the identities \eqref{xp} and \eqref{charfun}. From this perspective
it differs from the typical proofs of asymptotic normality in the framework of discrete moving average  
(cf.\ \cite{hh97,h99,s04}). We also remark that their conditions, translated to continuous time L\'evy moving 
averages,
are not satisfied in our setting. \qed   
\end{rem}
The asymptotic theory of this paper has a variety of potential applications and extensions. Let us first remark that Theorem 
\ref{maintheorem} uniquely identifies the parameters $\alpha$ and $\beta$. Notice that the convergence
rates  of Theorem~\ref{maintheorem}(i)--(iii) are all different under the corresponding conditions. Indeed, it holds that
\[
p(\alpha +1/\beta)-1<\alpha p< pk -1,
\]
since in case (i) we have $\alpha <k-1/p$ and in case (ii) we have $p<\beta$. Hence, computing the statistic $V(p;k)_n$ at log scale
for a continuous range of powers $p$  identifies the parameters $\alpha$ and $\beta$. More specifically, Theorem \ref{maintheorem}(ii) can be applied
directly to estimate the parameter $\alpha +1/\beta$. Indeed, we immediately obtain the convergence in probability
\[
\frac{\sum_{i=1}^{[n/2]} |X_{\frac {2i}n} - X_{\frac {2(i-1)}n}|^p}{\sum_{i=1}^n |X_{\frac in} - X_{\frac {i-1}n}|^p} \toop 2^{p(\alpha +1/\beta)-1}
\]
under conditions of Theorem \ref{maintheorem}(ii). Thus, a consistent estimator of
 $\alpha$ (resp. $\beta$) can be constructed given the knowledge
of $\beta$ (resp. $\alpha$) and the validity of conditions $\alpha \in (0,1-1/\beta)$ and $p<\beta$ 
(similar techniques remain valid for an arbitrary $k\geq 1$). We refer to a recent work \cite{glt15}, 
which applies
 log statistics of linear fractional stable motion to estimate parameters $\alpha$ and $\beta$. 

The results of Theorems \ref{maintheorem} and \ref{sec-order} can be extended in various directions. One may prove functional convergence for a 
partial sums formulation of the statistic $V(p;k)_n$ (the reformulation of asymptotic results is then rather obvious). For instance, we immediately
obtain uniform convergence in probability in Theorem \ref{maintheorem}(ii) and (iii), because the limits are continuous in time and the statistic is increasing
in time. It is likely to deduce the weak convergence towards a Brownian motion, resp.\ symmetric $(1-\alpha)\beta$-stable L\'evy process, in Theorem
\ref{sec-order}(ii), resp.\ (i) (cf.\ e.g.\  \cite{s04} for functional limit theorems in the discrete case). 
The case of Theorem \ref{maintheorem}(i) might be more complicated to handle.     

In another direction the limit theory may well be extended to integrals with respect to stationary increments L\'evy moving averages (see e.g.\  \cite{sg14} for some related results) or to the 
so called ambit processes, which include an additional multiplicative random input in the definition of the model \eqref{def-of-X-43} (see \cite{bs07}
for the definition, properties and applications of ambit processes). In this context the Bernstein's blocking technique is usually used
to extend Theorems \ref{maintheorem} and \ref{sec-order} to a more complex setting.

\section{Preliminaries} \label{secPrel}

Throughout the following sections all positive constants will be denoted by $K$, although they may change from line to line. 
Also the notation might change from subsection to subsection, but the meaning will be clear from the context. Throughout all the next sections we assume, without loss of generality,  that $c_0=\delta=\sigma= 1$. Recall that  $g(t)=g_0(t)=0$ for all $t<0$ by assumption.

For a sequences of random variables  $(Y_n)_{n\in \N}$ defined on the probability space $(\Omega,\mathcal F,\P)$ 
we write $Y_n\stab Y$ if $Y_n$ converges $\f$-stably in law to $Y$, i.e.\ $Y$ is a random variable defined on an extension  of $(\Omega,\mathcal F,\P)$ such that 
for all $\f$-measurable random variables $U$ we have the joint convergence in law  $(Y_n,U)  \schw (Y,U)$. 
In particular, $Y_n\stab Y$ implies $Y_n\schw Y$.  For $A\in \f$ we will say that $Y_n\stab Y$ on $A$, if $Y_n\stab Y$ under $\P_{|A}$, where  $\P_{|A}$ denotes the conditionally probability measure $B\mapsto \P(B \cap A)/\P(A)$, when $\P(A)>0$.  We refer to the work \cite{Aldous,ren} for a detailed exposition of stable convergence. In addition,  $\toop$ will denote convergence in probability. We will  write $V(Y,p;k)_n=\sum_{i=k}^n |\Delta^n_{i,k} Y|^p$ when we want to stress that the power variation is built from a process $Y$. 
On the other hand, when  $k$ and $p$ are fixed we will sometimes write $V(Y)_n=V(Y,p;k)_n$ to simplify the notation. 

For all $n,i\in \N$ set
\begin{align}\label{def-g-i-n}
g_{i,n}(x) ={}&  \sum_{j=0}^k (-1)^j \binom{k}{j} g\big((i-j)/n-x\big),\\ 
\label{def-h-i-n}
h_{i,n}(x) = {}& \sum_{j=0}^k (-1)^j \binom{k}{j} \big((i-j)/n-x\big)_+^\alpha,\\
\label{def-g-n}
g_n(x)={}& n^{\alpha}g(x/n), \qquad x\in \R.
\end{align}
In addition,  for each function $\phi\!:\R\to \R$ define  $D^k \phi\!:\R\to\R$  by 
\begin{align} \label{dkdef}
D^k \phi(x)=\sum_{j=0}^k (-1)^j \binom{k}{j} \phi(x-j),\qquad x\in \R.
\end{align}
In this notation the function $h_k$, defined in \eqref{def-h-13}, is given by $h_k=D^k \phi$ with $\phi: x\mapsto x_+^\alpha$.

%
%
%

\begin{lem} \label{helplem}
Assume that $g$ satisfies condition (A). Then we obtain the following estimates 
\begin{align}
\label{lemest1} |g_{i,n}(x)|&\leq K (i/n -x)^{\alpha}, \qquad x \in [(i-k)/n,i/n], \\
\label{lemest2} |g_{i,n}(x)|&\leq K n^{-k}((i-k)/n -x)^{\alpha-k} , \qquad x \in (i/n - 1, (i-k)/n), \\
\label{lemest3} |g_{i,n}(x)|&\leq K n^{-k}
\left( \1_{[(i-k)/n - 1, i/n - 1]}(x)  +
g^{(k)} ((i-k)/n -x) \1_{(-\infty, (i-k)/n - 1)} (x)  \right),\qquad  \\ 
& x \in (- \infty, i/n - 1].
\end{align}
The same estimates trivially hold for the function $h_{i,n}$. 
\end{lem}

\begin{proof}
The inequality \eqref{lemest1} follows directly from condition \eqref{kshs} of (A). The second inequality
\eqref{lemest2} is a straightforward consequence of Taylor expansion of order $k$ and the condition
$|g^{(k)}(t)|\leq K t^{\alpha-k}$ for $t\in (0,1)$. The third inequality \eqref{lemest3} follows again through Taylor expansion and the fact that the function $g^{(k)}$ is decreasing on $(1, \infty)$. 
\end{proof}

\section{Proof of Theorem \ref{maintheorem}}\label{proofs-w3lkhj}
\setcounter{equation}{0}
\renewcommand{\theequation}{\thesection.\arabic{equation}}

In this section we will prove the assertions of Theorem \ref{maintheorem}.

\subsection{Proof of Theorem~\ref{maintheorem}(i)} 

The proof of Theorem~\ref{maintheorem}(i) is divided into the following three steps.
In Step~(i) we show Theorem~\ref{maintheorem}(i) for the compound Poisson case, which stands for the treatment of big jumps of $L$. 
Step (ii) consists of an approximating lemma, which proves that the small jumps of $L$ are asymptotically negligible.
Step (iii) combines the previous results to obtain the general theorem. 

Before proceeding with the proof  we will need the following preliminary lemma. Let $\{x\}:=x-\lfloor x\rfloor\in [0,1)$ denote the fractional
part of $x\in \R$. The lemma below seems to be essentially known (cf.\ \cite{Jacod-round-off,Tukey}), however, we have not been able 
to find this particular formulation. Therefore it is stated below for completeness. 

\begin{lem}\label{frac-part-lem}
For $d\geq 1$ let $V=(V_1,\dots,V_d)$ be an absolutely continuous random vector in $\R^d$ with a density $v\!:\R^d\to \R_+$. Suppose that there exists an open convex set $A\subseteq \R^d$ such that $v$ is continuous differentiable on $A$ 
and vanish outside $A$.  Then, as $n\to \infty$, 
\begin{equation}\label{sflj}
\big(\{n V_1\},\dots,\{n V_d\}\big)\stab U=\big(U_1,\dots,U_d\big)
\end{equation}
 where $U_1,\dots,U_d$ are independent   $\mathcal U([0,1])$-distributed random variables which are independent of $\f$. 
\end{lem}

\begin{proof}
For $x=(x_1,\dots,x_d)\in \R^k$ let $\{x\}=(\{x_1\},\dots,\{x_d\})$ be the fractional parts of its components. 
Let  $f:\R^d\times \R^d\to \R$ be a $C^1$-function which vanish outside some closed ball  in $A\times \R^d$.  We claim that 
for all $\rho>0$ there exists a constant 
$K>0$ such that  
\begin{align}\label{eq-est-f-g}
D_\rho:= \Big| \int_{\R^d}  f(x,\{ x/\rho \})\, v(x) \,dx-\int_{\R^k}\Big(\int_{[0,1]^d} f(x,u)\,du\Big)v(x)\,dx\Big|  \leq   K \rho. 
\end{align}
Indeed, by \eqref{eq-est-f-g} used for $\rho=1/n$ we obtain that 
\begin{equation}\label{sdkfh}
\E[f(V,\{n V\})]\to \E[f(V,U)]\qquad \text{as }n\to \infty,
\end{equation}
with $U=(U_1,\dots, U_d)$ given in the lemma. Moreover, due to \cite[Proposition~2(D'')]{Aldous},  \eqref{sdkfh}  implies the stable convergence  $\{n V\}\stab U$ as $n\to\infty$,  and the proof  is  complete. Thus, it only remains to prove the inequality \eqref{eq-est-f-g}. At this stage we use a similar 
technique as in  \cite[Lemma~6.1]{Jacod-round-off}. 
 
Define   $\phi(x,u): =f(x,u)v(x)$. Then it holds by substitution that  
\begin{align}
\int_{\R^d}  f(x,\{  x/\rho \})v(x)\,dx= 
\sum_{j\in \Z^d} \int_{(0,1]^d} \rho^d \phi(\rho j + \rho u , u)\,du 
\end{align}
and
\begin{align}
\int_{\R^d}\Big(\int_{[0,1]^d} f(x,u)\,du\Big)v(x)\,dx={}& 
\sum_{j\in \Z^d}  \int_{[0,1]^d} \Big(\int_{(\rho j,\rho (j+1)]}\phi(x,u)\,dx\Big)\,du.
\end{align}
Hence, we conclude that  
\begin{align}
D_\rho\leq {}& \sum_{j\in \Z^d} \int_{(0,1]^d} \Big|\int_{(\rho j,\rho (j+1)]}\phi(x,u)\,dx-\rho^d \phi(\rho j + \rho u , u)\Big|\,du 
 \\  \leq {}& \sum_{j\in \Z^d} \int_{(0,1]^d} \int_{(\rho j,\rho (j+1)]}\Big|\phi(x,u)- \phi(\rho j + \rho u , u)\Big|\,dx\,du .
\end{align}
By mean value theorem there exists a positive constant $K$ and a compact set $B\subseteq \R^d\times \R^d$ such that for all $j\in \Z^d,\ x\in (\rho j,\rho(j+1)]$ and $u\in (0,1]^d$ we have 
\begin{equation}
\Big|\phi(x,u)- \phi(\rho j + \rho u , u)\Big|\leq K \rho \1_B(x,u). 
\end{equation}
Thus, 
$
D_\rho\leq K \rho \int_{(0,1]^d} \int_{\R^d}\1_{B}(x,u)\,dx\,du
$,
which shows \eqref{eq-est-f-g}. 
\end{proof}

\begin{proof}[Step~(i): The compound Poisson case]
Let $L=(L_t)_{t\in \R}$ be a compound Poisson process
and let $0\leq T_1<T_2<\ldots$ denote the jump times of the L\'evy process 
$(L_t)_{t\geq 0}$ chosen in increasing order.  Consider a fixed  $\epsilon>0$ and let $n\in \N$ satisfy $\varepsilon >n^{-1}$. We define 
\begin{align}
\Omega_{\varepsilon}:= \Big\{\omega \in \Omega : {}& \text{for all }k\geq 1\text{ with $T_k(\omega)\in [0,1]$ we have  }|T_k(\omega )-T_{k-1}(\omega)|> \varepsilon/2 \\ {}& \text{and }\Delta L_s=0\text{ for all }s\in [-\epsilon,0]\Big\}. 
\end{align}
Notice that $\Omega_{\varepsilon} \uparrow\Omega$ as  $\varepsilon \downarrow 0$. Now, we decompose for $i=k,\dots,n$
\[\label{rep-X-42}
\Delta_{i,k}^n X = M_{i,n,\varepsilon} + R_{i,n,\varepsilon}, 
\]
where 
\begin{align*}
M_{i,n,\varepsilon} = {}& 
\int_{\frac{i}{n}-\varepsilon /2}^{\frac{i}{n}} g_{i,n}(s)\,dL_s, \qquad \qquad R_{i,n,\varepsilon} = \int_{-\infty }^{\frac{i}{n}-\varepsilon /2} g_{i,n}(s)\,dL_s,
\end{align*}
and the function  $g_{i,n}$ is introduced in  \eqref{def-g-i-n}.
 The term $M_{i,n,\varepsilon}$ represents the dominating quantity, while $R_{i,n,\varepsilon}$ turns out to be negligible.

\noindent \textit{The dominating term:}
We claim that on $\Omega_\epsilon$ and as $n\to \infty$,
\begin{equation}\label{main-con}
 n^{\alpha p} \sum_{i=k}^n |M_{i,n,\epsilon}|^p\stab Z\qquad\text{where}\qquad Z=\sum_{m:\, T_m\in (0,1]} |\Delta L_{T_m}|^p V_m,
\end{equation}
where $V_m$, $m\geq 1$, are  defined in \eqref{part1}.
To show \eqref{main-con} let  $i_m=i_m(\omega,n)$ denote the (random) index such that $T_m\in ((i_m-1)/n, i_m/n]$. The following representation will be crucial:  On $\Omega_{\varepsilon}$ we have that
\begin{align}
\label{cru-decomp}
{}& n^{\alpha p} \sum_{i=k}^n | M_{i,n,\varepsilon}|^p =  V_{n,\varepsilon}\qquad\qquad \text{with}\\
{}&  V_{n,\varepsilon}  = n^{\alpha p} \sum_{m:\,T_m\in (0,1]} |\Delta L_{T_m}|^p \left(
\sum_{l=0}^{[\varepsilon n/2]+v_m} |g_{i_m+l,n}(T_m)|^p \right)
\end{align}
where the random index $v_m=v_m(\omega,n,\epsilon)$ are given by $v_m=0$ if $([\epsilon n/2]+i_m)/n-\epsilon/2< T_m$ and $v_m=-1$ else.
Indeed,  on $\Omega_\epsilon$ and for each $i=k,\dots,n$, $L$ has at most one jump in  $(i/n-\epsilon/2,i/n]$. For each $m\in \N$ with $T_m\in (0,1]$ we have  $T_m\in (i/n-\epsilon/2,i/n]$ if and only if $i\in \{i_m,\dots,[\epsilon n/2]+i_m+v_m\}$ (recall that $\epsilon>n^{-1}$). Thus,
\begin{equation}\label{sldfj}
\sum_{i\in \{k,\dots,n\}:\, T_m\in (i/n-\epsilon/2,i/n]} | M_{i,n,\varepsilon}|^p=|\Delta L_{T_m}|^p\left(\sum_{l=0}^{[\varepsilon n/2]+v_m} |g_{i_m+l,n}(T_m)|^p \right)
\end{equation}
and by summing \eqref{sldfj} over all  $m\in \N$ with $T_m\in (0,1]$,   \eqref{cru-decomp} follows. In the following we will show that 
\begin{equation}\label{con-Z-V-n}
V_{n,\epsilon} \stab Z
\qquad \text{as }n\to\infty. 
\end{equation}
For $d\geq 1$ it is well-known that  the random vector 
$(T_1,\dots,T_d)$ is absolutely continuous with a $C^1$-density   
on  the open convex set $A:=\{(x_1,\dots,x_d)\in \R^d\!: 0<x_1<x_2<\dots<x_d\}$  which is vanishing outside $A$, and thus, by Lemma~\ref{frac-part-lem}  we have  
\begin{equation}\label{eq-con-T_k-U}
(\{nT_m\})_{m\leq d}\stab (U_m)_{m\leq d}\qquad \text{as } n\to \infty
\end{equation}
where  
$(U_i)_{i\in \N}$ are  i.i.d.\ $\mathcal U([0,1])$-distributed random variables. By \eqref{kshs} we may write  $g(x)=x^\alpha_+ f(x)$  where $f\!:\R\to\R$  satisfies $f(x)\to 1$ as $x\downarrow 0$. 
By definition of $i_m$ we have that $\{n T_m\}=i_m-nT_m$ and therefore  
for all $l=0,1,2,\dots$ and $j=0,\dots,k$,  
\begin{align}\label{eq-sdlfj-1}
  {}&  n^\alpha g\Big(\frac{l+i_m-j}{n}-T_m\Big)= n^\alpha \Big(\frac{l+i_m-j}{n}-T_m\Big)^\alpha_+ f\Big(\frac{l+i_m-j}{n}-T_m\Big)\\ \label{eq-sdlfj-2}  {}& \qquad  =\Big(l-j+(i_m-nT_m)\Big)^\alpha_+ f\Big(\frac{l-j}{n}+n^{-1}(i_m-nT_m)\Big)
  \\ {}&\qquad = \Big(l-j+\{nT_m\}\Big)^\alpha_+ f\Big(\frac{l-j}{n}+n^{-1}\{nT_m\}\Big),
\end{align}
which by  \eqref{eq-con-T_k-U} and $f(x)\to 1$ as $x\downarrow 0$  shows that  
\begin{equation}\label{con-g-U}
\Big\{n^\alpha g\Big(\frac{l+i_m-j}{n}-T_m\Big)\Big\}_{l,m\leq d}
\stab \Big\{\big(l-j+U_m\big)^\alpha_+\Big\}_{l,m\leq d}  \qquad \text{as } n\rightarrow \infty. 
\end{equation}
 Eq.~\eqref{con-g-U} implies that 
\begin{align}\label{stab-con-g-23}
\big\{n^\alpha  g_{i_m+l,n}(T_m)\big\}_{l,m\leq d} \stab \big\{h_k(l+U_m)\big\}_{l,m\leq d},
\end{align}
with $h_k$ being defined at \eqref{def-h-13}. Due to the $\f$-stable convergence in \eqref{stab-con-g-23}  we obtain 
by the continuous mapping theorem that for each fixed $d\geq 1$  and as $n\to \infty$, 
\begin{align}
{}& V_{n,\epsilon,d}:=n^{\alpha p} \sum_{m:\,m\leq d,\,T_m\in [0,1]} |\Delta L_{T_m}|^p \left(
\sum_{l=0}^{[\varepsilon d/2]+v_m} |\Delta_{i_m+l,k}^ng(T_m)|^p \right)
\\ {}& \qquad  \stab Z_{d}=\sum_{m:\, m\leq d,\, T_m\in [0,1]} |\Delta L_{T_m}|^p \left( \sum_{l=0}^{[\varepsilon d/2]+v_m} | h_k(l+U_m)|^p \right).\label{eq-erlkj}
\end{align}
Moreover, for $\omega\in \Omega$ we have as $d\to \infty$,
\begin{equation}
Z_{d}(\omega)\uparrow Z(\omega).
\end{equation}
Recall that  $|h_k(x)|\leq K (x-k)^{\alpha-k}$ for $x>k+1$, which implies that $Z<\infty$ a.s.\ since $p(\alpha-k)<-1$.
For all $l\in \N$ with $k\leq l\leq n$, we have 
\begin{equation}\label{eq:7365}
n^{\alpha p } |g_{i_m+l,n}(T_m)|^p \leq K  | l-k| ^{(\alpha-k)p },
\end{equation}
due to \eqref{lemest2} of Lemma \ref{helplem}. 
For all $d\geq 0$ set  $C_d=\sum_{m>d:\, T_m\in [0,1]} |\Delta L_{T_m}|^p$ and note that $C_d\to 0$ a.s.\ as $d\to \infty$ since $L$ is a compound Poisson process. 
By \eqref{eq:7365} we have  
\begin{align}
 |V_{n,\epsilon}- V_{n,\epsilon,d}|
\leq {}& K\Big( C_d+ C_0 \sum_{l=[\varepsilon d/2]-1}^{\infty} |l-k|^{p(\alpha-k)}\Big) \to 0 \qquad \text{as }d\to \infty 
\end{align}
since $p(\alpha-k)<-1$. 
Due to the fact that  $n^{\alpha p} \sum_{i=k}^n | M_{i,n,\varepsilon}|^p=V_{n,\epsilon}$ a.s.\ on $\Omega_\epsilon$ and 
$V_{n,\epsilon}\stab Z$, it follows that $n^{\alpha p} \sum_{i=k}^n | M_{i,n,\varepsilon}|^p\stab Z$ on $\Omega_\epsilon$, since $\Omega_\epsilon\in \f$. This proves \eqref{main-con}.

\noindent 
\textit{The rest term:}
In the following we will show that 
\begin{equation}\label{rest-1}
n^{\alpha p} \sum_{i=k}^n | R_{i,n,\epsilon}|^p\toop 0\qquad \text{as } n\to \infty. 
\end{equation}
 The fact that  the random variables in \eqref{rest-1} are  usually not integrable makes the proof of \eqref{rest-1} considerable   more complicated.
Similar as in  \eqref{lemest3} of Lemma~\ref{helplem} we have that 
\begin{equation}\label{est-g-34}
n^k |g_{i,n}(s)| \1_{\{s\leq i/n-\epsilon\}}\leq K\big(\1_{\{s\in [-1,1]\}}+
\1_{\{s<-1\}} |g^{(k)}(-s)|\big)=:\psi(s) 
\end{equation}
where $K=K_\epsilon$. 
We will use the function $\psi$ several times in the proof of  \eqref{rest-1}, which will be divided into the two special cases $\theta\in (0,1]$ and $\theta\in (1,2]$. 

Suppose first that  $\theta\in (0,1]$. 
To show \eqref{rest-1} it suffices to show that 
\begin{equation}\label{we-want-52-1}
\sup_{n\in \N,\, i\in \{k,\dots,n\}} n^k|R_{i,n,\epsilon}|<\infty\qquad \text{a.s.}
\end{equation}
 since $\alpha<k-1/p$. 
 To show \eqref{we-want-52-1} 
we will  first prove that 
\begin{equation}\label{lshlfsjl}
\int_\R \int_{\R} \Big(|\psi(s)x|\wedge 1\Big)\,\nu(dx)\,ds<\infty.
\end{equation}
Choose $\tilde K$ such that $\psi(x)\leq \tilde K$ for all $x\in \R$. For $u\in [-\tilde K,\tilde K]$ we have that 
\begin{align} 
\int_{\R} \Big(|u x|\wedge 1\Big)\,\nu(dx) \leq   {}&   K \int_1^\infty \Big( |xu|\wedge 1 \Big)x^{-1-\theta}\,dx
\\ 
\leq {}& 
\begin{cases}
K |u|^\theta   & \theta\in (0,1) \\ 
K |u|^\theta\log(1/u)  & \theta =1, 
\end{cases}\label{lsfjdslfhl}
 \end{align}
 where we have used that $\theta\leq 1$. 
By \eqref{lsfjdslfhl} applied 
to $u=\psi(s)$ and assumption (A) it follows that \eqref{lshlfsjl} is satisfied. 
Since $L$ is a symmetric compound Poisson process we can find a Poisson random measure $\mu$ with compensator $\lambda \otimes \nu$ such that for all $-\infty<u<t<\infty$, $L_t-L_u=\int_{(u,t]\times \R} x\,\mu(ds,dx)$.
Due to \cite[Theorem~10.15]{k83}, \eqref{lshlfsjl} ensures the existence  of the stochastic 
integral $\int_{\R\times \R} |\psi(s)x| \,\mu(ds,dx)$. Moreover, $\int_{\R\times \R} |\psi(s)x| \,\mu(ds,dx)$ can be regarded as an $\omega$ by $\omega$  integral with respect to the measure $\mu_\omega$. 
Now, we have that 
\begin{align}\label{eq:883}
 | n^k R_{i,n,\epsilon} |\leq      \int_{(-\infty,i/n-\epsilon]\times \R}  \big| n^k g_{i,n}(s)x \big|\, \mu(ds,dx)  \leq \int_{\R\times \R} |\psi(s)x| \,\mu(ds,dx)<\infty, \qquad 
\end{align} 
which shows \eqref{we-want-52-1}, since the right-hand side of \eqref{eq:883} does not depend on $i$ and $n$.  
 
Suppose that  $\theta\in (1,2]$. 
Similarly as  before it suffices to show that 
\begin{equation}\label{we-want-52}
\sup_{n\in \N,\, i\in \{k,\dots,n\}} \frac{n^k|R_{i,n,\epsilon}| }{(\log n)^{1/q}}<\infty\qquad \text{a.s.}
\end{equation}
where $q>1$ denotes the conjugated number to $\theta>1$ determined by $1/\theta+1/q=1$.  
In the following we will show \eqref{we-want-52}  using the majorizing measure techniques developed in \cite{Marcus-Rosinski}. In fact, our  arguments are closely related to their Section~4.2.   
Set $T=\{(i,n)\!: n\geq k,\,i=k,\dots, n\}$. For $(i,n)\in T$ we have 
\begin{equation}
\frac{n^k|R_{i,n,\epsilon}| }{(\log n)^{1/q}}=\Big|\int_\R \zeta_{i,n}(s)\,dL_s\Big|, \qquad 
\zeta_{i,n}(s):=\frac{n^k}{(\log n)^{1/q}} g_{i,n}(s)\1_{\{s\leq i/n-\epsilon\}}.
\end{equation}
For $t=(i,n)\in T$ we will sometimes write $\zeta_t(s)$ for $\zeta_{i,n}(s)$. 
Let  $\tau\!: T\times T\to \R_+$ denote the metric given by 
\begin{equation}
\tau\big((i,n),(j,m)\big)=\begin{cases}
\log(n-k+1)^{-1/q}+\log(m-k+1)^{-1/q} & (i,n)\neq (j,l) \\ 0 & (i,n)=(j,l). 
\end{cases}
\end{equation}
Moreover, let $m$ be the probability measure on $T$ given by  $m(\{(i,n)\})=K n^{-3}$ for a suitable constant $K>0$.   
Set $B_\tau(t,r)=\{s\in T\!: \tau(s,t)\leq r\}$ for $t\in T$, $r>0$,  $D=\sup\{\tau(s,t)\!: s,t\in T\}$ and 
\begin{align}
I_q(m,\tau;D)=\sup_{t\in T}\int_0^D \Big(\log \frac{1}{m(B_\tau(t,r))}\Big)^{1/q} dr.
\end{align}
In the following we will show that $m$ is a so-called majorizing measure, which means that  
$I_q(m,\tau,D)<\infty$. 
For $r< (\log (n-k+1))^{-1/q}$ we have $B_\tau((i,n),r)=\{(i,n)\}$. Therefore, $m( B_\tau((i,n),r))=Kn^{-3}$ and 
\begin{align}\label{slfj-2}
 \int_0^{(\log(n-k+1))^{-1/q}}  \Big(\log \frac{1}{m(B_\tau((i,n),r))}\Big)^{1/q} dr
 = \int_0^{(\log(n-k+1))^{-1/q}} \Big(3\log n+\log K\Big)^{1/q}\,dr.
\end{align} 
For all $r\geq (\log(n-k+1))^{-1/q}$, $(k,k)\in B_\tau((i,n),r)$  and hence $m(B_\tau((i,n),r))\geq m(\{(k,k)\})=K (k+1)^{-3}$. Therefore, 
\begin{align}\label{slfj-3}
{}& \int_{(\log(n-k+1))^{-1/q}}^D  \Big(\log \frac{1}{m(B_\tau((i,n),r))}\Big)^{1/q} \,dr 
\\ {}& \qquad 
\leq
\int_{(\log(n-k+1))^{-1/q}}^D  \Big(3\log(k+1)+\log K\Big)^{1/q}\,dr.
\end{align}
By \eqref{slfj-2} and \eqref{slfj-3} it follows that 
$I_q(m,\tau,D)<\infty$. 
For  $(i,n)\neq (j,l)$ we have that 
\begin{align}\label{est-g-35}
\frac{| \zeta_{i,n}(s)-\zeta_{j,l}(s)|}{\tau\big((i,n),(j,l)\big)}\leq {}& 
n^k |g_{i,n}(s)| \1_{\{s\leq i/n-\epsilon\}}+l^k |g_{j,l}(s)| \1_{\{s\leq j/l-\epsilon\}}\leq K \psi(s).
\end{align}
Fix $t_0\in T$ and consider the following Lipschitz
type norm of $\zeta$, 
\begin{equation}
\|\zeta \|_\tau(s)= D^{-1}|\zeta_{t_0}(s)|+ \sup_{\substack{t_1,t_2\in T:\\ \, \tau(t_1,t_2)\neq 0}} \frac{|\zeta_{t_1}(s)-\zeta_{t_2}(s)|}{\tau(t_1,t_2)}.
\end{equation}
By \eqref{est-g-35} it follows  that 
$\| \zeta \|_{\tau}(s)\leq K  \psi(s) $
and hence   
\begin{equation}\label{eq-tau-norm}
\int_\R \|\zeta \|^\theta_\tau(s)\,ds\leq K\Big(2+\int_1^\infty |g^{(k)}(s)|^\theta\,ds\Big)<\infty.
\end{equation}
By \cite[Theorem~3.1, Eq.~(3.11)]{Marcus-Rosinski} together with $I_q(m,\tau,D)<\infty$ and \eqref{eq-tau-norm} we deduce  \eqref{we-want-52}, which completes the proof of \eqref{rest-1}.

\noindent
\textit{End of the proof:}
 Recall the decomposition $\Delta^n_{i,n}X= M_{i,n,\epsilon}+R_{i,n,\epsilon}$ in  \eqref{rep-X-42}. Eq.~\eqref{main-con}, \eqref{rest-1} and an application of Minkowski inequality yield that  
 \begin{equation}\label{eqyert}
n^{\alpha p} V(p;k)_n\stab Z\qquad \text{ on $\Omega_\epsilon$ as $n\to \infty$.}
 \end{equation}
 Since $\Omega_\epsilon\uparrow \Omega$ as $\epsilon\to 0$, \eqref{eqyert} implies that 
 \begin{equation}
n^{\alpha p} V(p;k)_n\stab  Z. 
 \end{equation}
We have now completed the proof for a particular  choice 
of stopping times $(T_m)_{m\geq 1}$, however, the result remains valid for any choice of $\F$-stopping times, since the distribution 
of $Z$ is invariant with respect to reordering of stopping times. 
\end{proof}

\noindent
{\em Step~(ii): An approximation.}
 To prove Theorem~\ref{maintheorem}(i) in the general case we need the following approximation result.
Consider a general symmetric L\'evy process $L=(L_t)_{t\in \R}$ as in Theorem~\ref{maintheorem}(i)
and let $N$ be the corresponding Poisson random measure $N(A):= \sharp\{t: (t,\Delta L_t)\in A\}$ for all measurable  $A\subseteq \R\times (\R\setminus\{0\})$.
%
%
By our assumptions (in particular, by  symmetry),  the  process $X(j)$ given by 
\begin{equation}\label{def-X-jaj}
X_t(j)=\int_{(-\infty,t]\times [-\frac{1}{j},\frac{1}{j}]} \big\{(g(t-s)-g_0(-s))x \big\}\,N(ds,dx)
\end{equation}
is well-defined. The following estimate on the processes $X(j)$ will be  crucial:

\begin{lem}\label{lem-1}
Suppose that   $\alpha<k-1/p$  and  $\beta<p$. 
Then  
\begin{equation*}
\lim_{j\to \infty} \limsup_{n\to \infty}\P\big(n^{\alpha p}V(X(j))_n>\epsilon\big)=0\qquad \text{for all }\epsilon>0.
\end{equation*}
\end{lem}

\begin{proof}
 By Markov's inequality and the stationary increments of $X(j)$ we have that 
\begin{align}
\P\big(n^{\alpha p}V(X(j))_n>\epsilon\big)
 \leq \epsilon^{-1}n^{\alpha p}  \sum_{i=k}^n \E [|\Delta^n_{i,k} X(j)|^p]\leq  \epsilon^{-1}n^{\alpha p +1} \E[|\Delta_{k,k}^n X(j)|^p].
\end{align}
Hence it is enough to show that 
\begin{equation}\label{con-Y}
\lim_{j\to \infty} \limsup_{n\to \infty}  \E[|Y_{n,j} |^p] = 0\qquad \text{with}\quad Y_{n,j}:=n^{\alpha +1/p} \Delta^n_{k,k} X(j).
\end{equation}
To show \eqref{con-Y} it sufficers  to show  
\begin{align}
\label{def-h-2}
{}& \lim_{j\to \infty} \limsup_{n\to \infty} \xi_{n,j}= 0\qquad \text{where}\qquad   \xi_{n,j}=   \int_{|x|\leq 1/j} \chi_n(x)\,\nu(dx)\quad \text{and} \\ \label{def-h}
{}& \chi_n(x)=\int_{-\infty}^{k/n} \Big(|n^{\alpha+1/p} g_{k,n}(s)x|^p\1_{\{|n^{\alpha+1/p} g_{k,n}(s)x |\geq 1\}}
\\ {}& \phantom{\chi_n(x)=\int_{-\infty}^{k/n} \Big(  }
+|n^{\alpha+1/p} g_{k,n}(s)x|^2\1_{\{|n^{\alpha+1/p} g_{k,n}(s)x |\leq 1\}}\Big)\,ds,
\end{align}
which follows from the  representation  
\begin{equation*}
Y_{n,j}=\int_{(-\infty,k/n]\times  [-\frac{1}{j},\frac{1}{j}]}\big(n^{\alpha+1/p} g_{k,n}(s)x\big)\, N(ds,dx),
\end{equation*}
and   by   \cite[Theorem~3.3 and the remarks  above it]{RajRos}. 
Suppose for the moment that there exists a finite constant $K>0$ such that 
\begin{equation}\label{est-h-ewrlj}
\chi_n(x)\leq  K (|x|^p+x^2) \qquad \text{for all } x\in [-1,1].
\end{equation}
Then, 
\begin{equation}\label{last-est-sldfj}
\limsup_{j\to\infty}\big\{\limsup_{n\to \infty} \xi_{n,j}\big\}\leq K \limsup_{j\to\infty}\int_{|x|\leq 1/j} (|x|^p+x^2)\,\nu(dx)=0
\end{equation}
since $p>\beta$. Hence it suffices to show the estimate \eqref{est-h-ewrlj}, which we will do in the following.

Let $\Phi_p\!:\R\to\R_+$ denote the function  $\Phi_p(y)= |y|^2\1_{\{|y|\leq 1\}} +|y|^p\1_{\{|y|> 1\}}$.
We split $\chi_n$ into the following three terms which need different treatments
 \begin{align}
  \chi_n(x) {}& = \int_{-k/n}^{k/n} \Phi_p\Big( n^{\alpha+1/p} g_{k,n}(s)x\Big)\,ds +\int_{-1}^{-k/n} \Phi_p\Big(n^{\alpha+1/p} g_{k,n}(s)x\Big)\,ds \\{}
 {}&\phantom{=}+\int_{-\infty}^{-1} \Phi_p\Big(n^{\alpha+1/p} g_{k,n}(s) x\Big)\,ds \\ {}
{}& =:   I_{1,n}(x)+I_{2,n}(x)+I_{3,n}(x).
 \end{align}
\emph{Estimation of $I_{1,n}$}: By \eqref{lemest1} of Lemma \ref{helplem} we have that 
\begin{equation}\label{rep-g-slfj}
|g_{k,n}(s)|\leq K (k/n-s)^\alpha, \qquad s\in [-k/n,k/n].
\end{equation}  Since  $\Phi_p$ is increasing on $\R_+$, 
 \eqref{rep-g-slfj} implies that  
\begin{align}\label{est-v_n-1}
I_{1,n}(x)\leq K \int_{0}^{2k/n}\Phi_p\Big(x n^{\alpha+1/p}s^\alpha\Big)\,ds.
\end{align} 
By basic calculus it follows that 
\begin{align}
{}& \int_{0}^{2k/n} |xn^{\alpha+1/p}s^\alpha |^2\1_{\{|xn^{\alpha+1/p}s^\alpha|\leq 1\}}\,ds \nonumber
\\  
{}& \quad\leq K \Big( \1_{\{|x|\leq (2k )^{-\alpha}n^{-1/p}\}} x^2 n^{2/p-1} + \1_{\{|x|>(2k )^{-\alpha}n^{-1/p}\}} |x|^{-1/\alpha} n^{-1-1/(\alpha p)}\Big)
\\ {}& \quad\leq K (|x|^p +x^2). \label{lhsflh}
\end{align}
Moreover, 
\begin{align}\label{est-h-v2}
{}&  \int_{0}^{2k/n} |x  n^{\alpha+1/p} s^\alpha|^p \1_{\{|x  n^{\alpha+1/p} s^\alpha|>1\}}\,ds\leq \int_{0}^{2k/n} |x  n^{\alpha+1/p} s^\alpha|^p 
  \,ds 
  \leq K |x|^p. 
\end{align}
Combining \eqref{est-v_n-1}, \eqref{lhsflh} and \eqref{est-h-v2}  show the estimiate 
$I_{1,n}(x)\leq K(|x|^p+x^2)$.
 
 \noindent
\emph{Estimation of $I_{2,n}$}:
By \eqref{lemest2} of Lemma \ref{helplem} it holds  that 
\begin{align}\label{est-ljsdflj}
|g_{k,n}(s)|\leq K n^{-k} |s|^{\alpha-k}, \qquad s\in (-1,-k/n). 
\end{align}
Again, due to the fact that $\Phi_p$ is increasing on $\R_+$, \eqref{est-ljsdflj} implies that 
\begin{align}\label{kjhsfg-231}
 I_{2,n}(x) \leq K \int_{k/n}^{1} \Phi_p(x n^{\alpha+1/p-k}
s^{\alpha-k})\,ds.
\end{align}
For  $\alpha\neq k-1/2$ we have  
\begin{align}
 {}& \int_{k/n}^{1} |xn^{\alpha+1/p-k} s^{\alpha-k}|^2\1_{\{|xn^{\alpha+1/p-k}s^{\alpha-k}|\leq 1\}}\,ds \nonumber
\\ {}&\quad \leq 
 K \Big(x^2 n^{2(\alpha+1/p-k)} 
+\1_{\{|x|\leq n^{-1/p}k^{-(\alpha-k)}\}} |x|^2 n^{2/p-1} 
 \\ {}& \qquad \qquad + 
 \1_{\{|x|> n^{-1/p}k^{-(\alpha-k)}\}}
|x|^{1/(k-\alpha)}n^{1/(p(k-\alpha))-1}\Big)
 \\ \label{lsdjflsh-2}{}& \quad \leq K \Big(x^2+|x|^p\Big)
\end{align}
where we have used that  $\alpha< k-1/p$. 
 For $\alpha=k-1/2$ we have 
\begin{align}
{}& \int_{k/n}^{1} |xn^{\alpha+1/p-k} s^{\alpha-k}|^2\1_{\{|xn^{\alpha+1/p-k}s^{\alpha-k}|\leq 1\}}\,ds \nonumber
\\ \label{lsdjflsh-4}{}&\qquad \leq 
x^2 n^{2(\alpha+1/p-k)} \int_{k/n}^1 s^{-1} \,ds=x^2 n^{2(\alpha+1/p-k)}\log(n/k)\leq  K x^2,
\end{align}
where we again have used   $\alpha<k-1/p$ in the last inequality. 
Moreover, 
\begin{align}\label{lsdjflsh-5}
{}& \int_{k/n}^{1} |xn^{\alpha+1/p-k} s^{\alpha-k}|^p\1_{\{|xn^{\alpha+1/p-k}s^{\alpha-k}|> 1\}}\,ds 
\\ {}& \qquad \leq 
K |x|^p n^{p(\alpha+1/p-k)} \Big(1+(1/n)^{p(\alpha-k)+1}\Big)\leq K |x|^p.  \label{lsdjflsh-6}
\end{align}
By \eqref{kjhsfg-231}, \eqref{lsdjflsh-2}, \eqref{lsdjflsh-4} and \eqref{lsdjflsh-6} we obtain the estimate $I_{2,n}(x)\leq K(|x|^p+x^2)$. 

\noindent
\emph{Estimation of $I_{3,n}$}:
For $s<-1$ we have that 
$
|g_{k,n}(s)|\leq K n^{-k} |g^{(k)}(-k/n-s)|,
$
by \eqref{lemest3} of Lemma \ref{helplem},
and hence 
\begin{align}\label{klhshlkslh}
{}&I_{3,n}(x)
\leq  K \int_{1}^{\infty} 
\Phi_p\big(n^{\alpha+1/p-k} g^{(k)}(s)\big)\,ds.
\end{align}
We have that 
\begin{equation}\label{iuyhswrhjl}
 \int_{1}^{\infty} |xn^{\alpha+1/p-k} g^{(k)}(s)|^2\1_{\{|xn^{\alpha+1/p-k}g^{(k)}(s)|\leq 1\}}\,ds
 \leq 
 x^2 n^{2(\alpha+1/p-k)} \int_{1}^\infty |g^{(k)}(s)|^2\,ds.
\end{equation}
Since $|g^{(k)}|$ is decreasing on $(1,\infty)$ and $g^{(k)}\in L^\theta((1,\infty))$ for some $\theta\leq 2$, the integral on the right-hand side of \eqref{iuyhswrhjl} is finite. For $x\in [-1,1]$ we have  
\begin{align}
{}& \int_{1}^{\infty} |xn^{\alpha+1/p-k} g^{(k)}(s)|^p\1_{\{|xn^{\alpha+1/p-k}g^{(k)}(s)|> 1\}}\,ds \nonumber
 \\ {}& \qquad \leq 
 |x|^p n^{p(\alpha+1/p-k)} \int_{1}^\infty |g^{(k)}(s)|^p\1_{\{|g^{(k)}(s)|> 1\}}\,ds. \label{skdhhfho}
\end{align}
From  our assumptions it follows that the integral in \eqref{skdhhfho} is finite. By \eqref{klhshlkslh}, \eqref{iuyhswrhjl} and \eqref{skdhhfho} 
we have  that 
$I_{n,3}(x)\leq K (|x|^p+x^2)$ for all $x\in [-1,1]$,
which completes the proof of \eqref{est-h-ewrlj} and therefore also the proof of the lemma.  
\end{proof}

\noindent
{\em Step~(iii): The general case.}
In the following we will prove Theorem~\ref{maintheorem}(i) in the general case  by combining the above Steps~(i) and (ii). 

 \begin{proof}[Proof of Theorem~\ref{maintheorem}(i)]
Let $(T_m)_{m\geq 1}$ be a sequence of $\F$-stopping times that exhaust the jumps of $(L_t)_{t\geq 0}$. For each $j\in \N$ let $\hat L(j)$ be the L\'evy process given by 
\begin{equation}
\hat L_t(j)-\hat L_u(j)=\sum_{u\in (s,t]} \Delta L_u\1_{\{|\Delta L_u|>\frac{1}{j}\}},\qquad s<t,
\end{equation}
and set 
\begin{equation}
\hat X_t(j)= \int_{-\infty}^t \big(g(t-s)-g_0(-s)\big)\,d\hat L_s(j). 
\end{equation}
Moreover, set 
\begin{equation}
T_{m,j}=\begin{cases}
T_{m} & \text{if }|\Delta L_{T_m}|>\frac{1}{j} \\ \infty & \text{else}, 
\end{cases}
\end{equation}
and note that  $(T_{m,j})_{m\geq 1}$ is a sequence of $\F$-stopping times that exhausts the jumps of  $(\hat L_t(j))_{t\geq 0}$.
Since  $\hat L(j)$ is a compound Poisson process,  Step~1 shows that 
\begin{align}\label{app-12}
n^{\alpha p}V(\hat X(j))_n \stab Z_j:=\sum_{m:\,T_{m,j}\in [0,1]} |\Delta \hat L_{T_{m,j}}(j)|^p V_{m} \qquad \text{as } n\rightarrow \infty,
\end{align}
where $V_m$, $m\geq 1$, are defined in \eqref{part1}. 
By definition of $T_{m,j}$ and monotone convergence we have as $j\to \infty$, 
\begin{equation}\label{app-22}
Z_j=\sum_{m:\,T_{m}\in [0,1]} |\Delta L_{T_{m}}|^p V_{m} \1_{\{|\Delta L_{T_m}|>\frac{1}{j}|\}} \toas \sum_{m:\,T_{m}\in [0,1]} |\Delta L_{T_{m}}|^p V_m=:Z.
\end{equation}
Suppose first that $p\geq 1$ and decompose
 \begin{align}
\big(n^{\alpha p} V(X)_n\big)^{1/p}{}& = \big(n^{\alpha p}V(\hat X(j))_n\big)^{1/p} +\Big(\big(n^{\alpha p}V(X)_n\big)^{1/p}-\big(n^{\alpha p}V(\hat X(j))_n\big)^{1/p}\Big)
\\ {}& =: Y_{n,j}+U_{n,j}.\label{eq:746}
\end{align}
Eq.~\eqref{app-12} and \eqref{app-22} show 
\begin{equation}\label{eq:3456}
Y_{n,j}\xrightarrow[n\to \infty]{\mathcal{L}-s} Z_j^{1/p}\qquad \text{and}\qquad Z_j^{1/p}\xrightarrow[j\to \infty]{\P} Z^{1/p}. 
\end{equation}
Note that $X-\hat X(j)=X(j)$, where $X(j)$ is defined in \eqref{def-X-jaj}. For all $\epsilon>0$ we have by Minkowski's inequality  
\begin{align}\label{eq:63653}
{}& \limsup_{j \to \infty}\limsup_{n\to \infty} \P\big( |U_{n,j}|>\epsilon\big)  \leq 
\limsup_{j \to \infty}\limsup_{n\to \infty} \P\big(n^{\alpha p}V(X(j))_n>\epsilon^p\big)=0
\end{align}
where the last equality follows by Lemma~\ref{lem-1}. By a standard argument, see e.g.\ \cite[Theorem~3.2]{Billingsley}, \eqref{eq:3456} and \eqref{eq:63653} implies that  
$(n^{\alpha p} V(X)_n)^{1/p}\stab Z^{1/p}$ which completes the proof of Theorem~\ref{maintheorem}(i) when $p\geq 1$. For  $p<1$, Theorem~\ref{maintheorem}(i) follows by \eqref{app-12}, \eqref{app-22}, the inequality  
$| V(X)_n- V(\hat X(j))_n |\leq V(X(j))_n$  and \cite[Theorem~3.2]{Billingsley}. 
\end{proof}

\subsection{Proof of Theorem~\ref{maintheorem}(ii)}

Suppose that $\alpha <k-1/\beta$, $p<\beta$ and $L$ is a symmetric $\beta$-stable L\'evy proces. In the proof of Theorem~\ref{maintheorem}(ii) we will use  
the process $V=(V_t)_{t\geq 0}$ given by 
\begin{equation}\label{def-pro-Ysf}
V_t=\int_{-\infty}^t h_k(t-s)\,dL_s
\end{equation}
to approximate the scaled version of  $k$-order increments of $X$.
For $\alpha<1-1/\beta$,  $Y$ is the $k$-order increments of a linear fractional stable motion. For $\alpha\geq 1-1/\beta$ the linear fractional stable motion is not well-defined, but $Y$ is well-defined since the function $h_k$ is locally bounded and satisfies  $|h_k(x)|\leq K x^{\alpha-k}$ for $x\geq k+1$,  
which implies that $h_k\in L^\beta(\R)$.

\begin{proof}[Proof of Theorem~\ref{maintheorem}(ii)]
By self-similarity of $L$ of index $1/\beta$ we have for all $n\in \N$, 
\begin{equation}\label{self-similar}
\{n^{\alpha+1/\beta} \Delta^n_{i,k} X\!:i=k,\dots,n\}\eqschw \{V_{i,n}\!:i=k,\dots,n\}
\end{equation}
where 
\begin{equation} \label{gndef} 
V_{i,n}= \int_{-\infty}^i    D^k g_{n}(i-s)\,dL_s,
\end{equation}
$g_n$ and $D^k$ are defined at \eqref{def-g-n} and \eqref{dkdef},  and $\eqschw$ means equality in distribution. In the following we will show that $V_{k,n}$ and $V_k$ are close in $L^p$ when $n$ is large, where process $V$ is  given by \eqref{def-pro-Ysf}. For $s\in \R$ let  
$\psi_n(s)=g_n(s)-s_+^{\alpha}$. Since $p<\beta$,   
\begin{align}\label{est-lkhgwkg}
\E[|V_{k,n}-Y_k|^p]=K \Big(\int_0^\infty | D^k \psi_n(s)  |^\beta \,ds\Big)^{p/\beta}.
\end{align}
To show that the right-hand side of \eqref{est-lkhgwkg} converge to zero we   note that 
\begin{align}\label{sdkkls}
{}& \int_{n+k}^\infty |D^k g_n(s)|^\beta\,ds \leq K n^{\beta(\alpha-k)} 
\int_{n+k}^\infty |g^{(k)}\big((s-k)/n\big)|^\beta\,ds
\\ \label{sdkkls-1} {}& \qquad = 
K n^{\beta(\alpha-k)+1} \int_{1 }^\infty |g^{(k)}(s)|^\beta\,ds\to 0\qquad \text{as }n\to \infty. 
\end{align}
Recall that for $\phi:s\mapsto s^\alpha_+$ we have $D^k \phi=h_k \in L^\beta(\R)$, which together with  \eqref{sdkkls}--\eqref{sdkkls-1}  show that  
\begin{align}\label{skdhks}
 \int_{n+k}^\infty |D^k \psi_n(s)|^\beta\,ds\leq {}& K\Big( \int_{n+k}^\infty |D^k g_n(s)|^\beta\,ds+\int_{n+k}^\infty |D^k s_+^{\alpha}|^\beta\,ds\Big)\xrightarrow[n\to\infty]{} 0.
\end{align}
By \eqref{lemest2} of Lemma \ref{helplem} it holds that 
\begin{equation}
|D^k g_n(s)|\leq  K(s-k)^{\alpha-k}
\end{equation}
for  $s\in (k+1,n)$.
Therefore,   for $s\in (0,n]$ we have 
\begin{equation}\label{dom-slfj}
|D^k\psi_n(s)|\leq K \big(\1_{\{s\leq k+1\}}+\1_{\{s>k+1\}}(s-k)^{\alpha-k}\big),  
\end{equation}
where  the function on the right-hand side of \eqref{dom-slfj} is in $L^{\beta}(\R_+)$. 
For fixed $s\geq 0$, $\psi_n(s)\to 0 $ as $n\to \infty$ by  assumption \eqref{kshs}, and hence $D^k \psi_n(s)\to 0$ as $n\to\infty$, which by \eqref{dom-slfj} and the dominated convergence theorem show that 
\begin{equation}\label{con-low}
\int_0^{ n} |D^k \psi_n(s)|^\beta\,ds\to 0. 
\end{equation}
By \eqref{est-lkhgwkg}, \eqref{skdhks} and \eqref{con-low}   we have $\E[|V_{k,n}-V_k|^p]\to 0 $ as $n\to \infty$, and hence 
  \begin{equation}\label{con-L-1}
\E\Big[ \frac{1}{n} \sum_{i=k}^n |V_{i,n}-Y_j|^p\Big]=  \frac{1}{n} \sum_{i=k}^n \E[|V_{i,n}-Y_i|^p]\leq \E[|V_{k,n}-Y_k|^p]\to 0 
 \end{equation} 
 as $n\to \infty$. 
Moreover,   $(V_t)_{t\in \R}$ is mixing since it is a symmetric stable moving average, see e.g.\ \cite{Ergodic-Cam}. This implies, in particular,  that the discrete time stationary sequence $\{Y_j\}_{j\in \mathbbm Z}$ is mixing and hence ergodic. According to Birkhoff's ergodic theorem 
\begin{equation}\label{Ergodic}
\frac{1}{n} \sum_{i=k}^n |V_i|^p \toas \E[|V_k|^p]=m_p \in (0,\infty)\qquad \text{as }n\to \infty,
\end{equation}
where $m_p$ has been defined in Theorem \ref{maintheorem}(ii). 
The equality  $m_p=\E[|V_k|^p]$  follows
 by \cite[Property~1.2.17 and 3.2.2]{SamTaq}.  
By  \eqref{con-L-1}, Minkowski's inequality and \eqref{Ergodic}, 
\begin{equation}\label{con-gsdfhk}
\frac{1}{n}\sum_{i=k}^n |V_{i,n}|^p\toop m_p \qquad \text{as }n\to\infty, 
\end{equation}
which  by  \eqref{self-similar} shows that 
\begin{align}
 n^{-1+p(\alpha+1/\beta)} V(X)_n={}& \frac{1}{n} \sum_{i=k}^n |n^{\alpha+1/\beta} \Delta^n_{i,k} X|^p   \eqschw \frac{1}{n}\sum_{i=k}^n |V_{i,k}|^p\toop m_p
\end{align}
as $n\to \infty$.  This completes the proof of Theorem~\ref{maintheorem}(ii).
\end{proof}

\subsection{Proof of Theorem~\ref{maintheorem}(iii)}

We will derive Theorem~\ref{maintheorem}(iii) from   the  two lemmas below. For $k\in \N$ and $p\in [1,\infty)$ let $W^{k,p}$ denote the Wiener space of functions $\zeta\!:[0,1]\to\R$ which are $k$-times absolutely continuous with  $\zeta^{(k)}\in L^p([0,1])$ where $\zeta^{(k)}(t)=\partial^k \zeta(t)/\partial t^k$ $\lambda$-a.s. First we  will show that, under the conditions in Theorem~\ref{maintheorem}(iii),    $X\in W^{k,p}$ a.s.

\begin{lem}\label{abs-cont-sdf}
Suppose  that $p\neq \theta$, $p\geq 1$ and (A).
 If $\alpha>k-1/(p\vee \beta)$ then  
 \begin{equation}\label{sdhjlshjl}
X\in W^{k,p}\text{ a.s.}\qquad \text{and}\qquad \frac{\partial^{k}}{\partial t^k} X_t=\int_{-\infty}^t g^{(k)}(t-s)\,dL_s\qquad \lambda\otimes \P\text{-a.s.}
 \end{equation}
 Eq.~\eqref{sdhjlshjl} remains valid for  $p=\theta$  if, in  addition, (A-log) holds. 
\end{lem}
\begin{proof}
We will not need the assumption \eqref{kshs} on $g$ in the proof.  For notation simplicity  we only consider the case $k=1$, since the general case follows by similar arguments.   
To prove \eqref{sdhjlshjl} it is sufficient 
to show that the three conditions (5.3), (5.4) and (5.6) from \cite[Theorem~5.1]{SamBra} are satisfied
(this result uses the condition $p\geq 1$). In fact, the representation \eqref{sdhjlshjl} of $(\partial/\partial t)X_t$ 
follows by the equation below (5.10) in  \cite{SamBra}.
In our setting the  function $\dot \sigma$ defined in \cite[Eq.~(5.5)]{SamBra} is constant and hence  (5.3), (5.4) and (5.6) in \cite{SamBra}  simplifies to 
\begin{align}
{}& \label{sldfhogh} \int_{\R} \nu\Big(\Big(\frac{1}{\| g'\|_{L^p([s,1+s])}},\infty\Big)\Big)\,ds<\infty,\\ 
{}& \label{sdflhkhgs} \int_0^\infty \int_\R \Big(|x g'(s)|^2\wedge 1\Big)\,\nu(dx)\,ds<\infty, \\ 
\label{khshjl}
{}& \int_0^1 \int_\R |g'(t+s)|^p\Big(\int_{r/|g'(t+s)|}^{1/\|g'\| _{L^p([s,1+s])}} x^p \,\nu(dx)\Big)\,ds\,dt<\infty
\end{align}
for all $r>0$. When the lower bound in the inner integral in \eqref{khshjl} exceed the upper bound  the integral is set to zero.
Since $\alpha>1-1/\beta$ we may choose $\epsilon>0$ such that $(\alpha-1)(\beta+\epsilon)>-1$. 
 To show \eqref{sldfhogh} we use the estimates
\begin{align}
\| g'\|_{L^p([s,1+s])}\leq {}& K \Big(\1_{\{s\in [-1,1]\}}  +\1_{\{s>1\}} |g'(s)|\Big),\qquad s\in \R,  
\intertext{and}
\nu((u,\infty))\leq {}& 
\begin{cases} K u^{-\theta} & u\geq 1 \\ 
K u^{-\beta-\epsilon} & u\in (0,1],
\end{cases}
\end{align}
which both follows from assumption (A). Hence, we deduce  
\begin{align}
{}& \int_{\R} \nu\Big(\Big(\frac{1}{\| g'\|_{L^p([s,1+s])}},\infty\Big)\Big)\,ds
\\ {}& \qquad \leq 
 \int_{-1}^1  \nu\Big(\Big(\frac{1}{K},\infty\Big)\Big)\,ds
 +  \int_{1}^{\infty}  \nu\Big(\Big(\frac{1}{K | g'(s)|},\infty\Big)\Big)\,ds
 \\ {}& \qquad \leq 
 2 \nu\Big(\Big(\frac{1}{K},\infty\Big)\Big)
 + K\int_{1}^{\infty}  \Big(|g'(s)|^\theta\1_{\{K |g'(s)|\leq 1\}}+|g'(s)|^{\beta+\epsilon}\1_{\{K|g'(s)|> 1\}}\Big)\,ds<\infty
\end{align}
which shows \eqref{sldfhogh} (recall that $|g'|$ is decreasing on $(1,\infty)$). To show  \eqref{sdflhkhgs} we  will use   the following two  estimates: 
\begin{align}\label{hgsdghl}
 {}&   \int_0^1 \Big(|s^{\alpha-1}x|^2\wedge 1\Big)\,ds
\leq 
\begin{cases} K\Big(\1_{\{|x|\leq 1\}} |x|^{1/(1-\alpha)}+  \1_{\{|x|> 1\}} \Big) & \alpha<1/2 \\ 
K\Big(\1_{\{|x|\leq 1\}} x^2\log(1/x)+  \1_{\{|x|>1\}} \Big)  & \alpha = 1/2 \\   K\Big(\1_{\{|x|\leq 1\}}x^2 +  \1_{\{|x|> 1\}} \Big) & \alpha>1/2,
\end{cases}
\intertext{and} 
\label{howeerho}
{}&  \int_{\{|x|>1\}}  \Big(|x g'(s)|^2\wedge 1\Big)\,\nu(dx) 
 \leq	  K \int_{1}^\infty  \Big(|x g'(s)|^2\wedge 1\Big) x^{-1-\theta}\,dx
 \leq K |g'(s)|^\theta.\qquad
\end{align}
For $\alpha<1/2$ we have 
\begin{align}
{}& \int_0^\infty \int_\R \Big(|x g'(s)|^2\wedge 1\Big)\,\nu(dx)\,ds\\
{}& \quad 
\leq K\Big\{\int_\R  \int_0^1 \Big(|x s^{\alpha-1}|^2\wedge 1\Big)\,ds\,\nu(dx)+
\int_1^\infty \int_{\{|x|\leq 1\}}   \Big(|x g'(s)|^2\wedge 1\Big)\,\nu(dx)\,ds  \\  {}&\qquad \qquad +\int_1^\infty   \int_{\{|x|> 1\}}  \Big(|x g'(s)|^2\wedge 1\Big)\,\nu(dx)\,ds\Big\}\\ {}&\quad 
 \leq K\Big\{ \int_\R  \big( \1_{\{|x|\leq 1\}} |x|^{1/(1-\alpha)}+\1_{\{|x|>1\}}\big)\, \nu(dx)\\ {}& \qquad \qquad +\Big(\int_1^\infty |g'(s)|^2\,ds\Big)\Big(\int_{\{|x|\leq 1\}} x^2\,\nu(dx) \Big) +\int_1^\infty |g'(s)|^\theta\,ds\Big\}<\infty, 
\end{align}
where  the first  inequality follows by  assumption (A), 
the second inequality follows by \eqref{hgsdghl} and \eqref{howeerho},  and   the last inequality is due to the fact that    $1/(1-\alpha)>\beta$ and $g'\in L^\theta((1,\infty))\cap L^2((1,\infty))$. 
This shows \eqref{sdflhkhgs}. The two remaining  cases $\alpha=1/2$ and $\alpha>1/2$ follow similarly.

Now, we will prove that \eqref{khshjl} holds.
Since $|g'|$ is decreasing on $(1,\infty)$ we have for all $t\in [0,1]$ 
that 
 \begin{align}\label{sdfkgkhsfh}
{}& \int_1^\infty |g'(t+s)|^p\Big(\int_{r/|g'(t+s)|}^{1/\|g'\| _{L^p([s,1+s])}} x^p \,\nu(dx)\Big)\,ds
\\ {}& \qquad  \leq 
 \int_1^\infty |g'(s)|^p\Big(\int_{r/|g'(1+s)|}^{1/|g'(s)|} x^p \,\nu(dx)\Big)\,ds
 \\ \label{khshf}{}& \qquad  \leq \frac{K}{p-\theta}
 \int_1^\infty |g'(s)|^p\Big(|g'(s)|^{\theta-p}-|g'(s+1)/r|^{\theta-p}\Big)\1_{\{r/|g'(1+s)|\leq 1/|g'(s)|\}}\,ds. \qquad 
  \end{align}
For $p>\theta$, \eqref{khshf} is less than or equal to 
  \begin{equation}
 \frac{K}{p-\theta} \int_1^\infty |g'(s)|^\theta\,ds<\infty,
  \end{equation}
 and for $p<\theta$,
 \eqref{khshf} is less than or equal to 
 \begin{align}
\frac{K r^{p-\theta}}{\theta-p}
 \int_1^\infty |g'(s)|^p |g'(s+1)|^{\theta-p}\,ds\leq 
 \frac{K r^{p-\theta}}{\theta-p} \int_1^\infty |g'(s)|^\theta\,ds<\infty,
 \end{align}
 where the first inequality is due to the fact that $|g'|$ is decreasing on $(1,\infty)$. 
 Hence we have shown that 
 \begin{equation}\label{khsk}
 \int_0^1 \int_1^\infty |g'(t+s)|^p\Big(\int_{r/|g'(t+s)|}^{1/\|g'\| _{L^p([s,1+s])}} x^p \,\nu(dx)\Big)\,ds\,dt<\infty
 \end{equation}
 for $p\neq \theta$. 
Suppose that $p>\beta$. For $t\in [0,1]$ and $s\in [-1,1]$ we have 
 \begin{align}\label{klhkl}
 {}& \int_{r/|g'(t+s)|}^{1/\|g'\| _{L^p([s,1+s])}} x^p \,\nu(dx)\leq  \int_{1}^{1/\|g'\| _{L^p([s,1+s])}} x^p \,\nu(dx)+ \int_{r/|g'(t+s)|}^{1} x^p \,\nu(dx)\qquad \\
 {}& \qquad \leq K\Big( \|g'\| _{L^p([s,1+s])}^{\theta-p}
+1 \Big)\label{klhkl-1}
 \end{align}
 and hence 
 \begin{align}\label{eq93e4u7}
  {}&  \int_0^1 \int_{-1}^1 |g'(t+s)|^p\Big(\int_{r/|g'(t+s)|}^{1/\|g'\| _{L^p([s,1+s])}} x^p \,\nu(dx)\Big)\,ds\,dt\\ \label{eq93e4u7-1}{}&\qquad  
   \leq K\Big(\int_{-1}^1 \|g'\|_{L^p([s,s+1])}^{\theta} \,ds+ \int_{-1}^1
    \|g'\|_{L^p([s,1+s])}^p\,ds\Big)<\infty.
 \end{align}
Suppose that  $p\leq \beta$. For  $t\in [0,1]$ and $s\in [-1,1]$ we have 
 \begin{align}
  \int_{r/|g'(t+s)|}^{1/\|g'\| _{L^p([s,1+s])}} x^p \,\nu(dx)\leq K\Big( \|g'\| _{L^p([s,1+s])}^{\theta-p}
+|g'(t+s)|^{\beta+\epsilon-p} \Big)
 \end{align}
 and hence 
 \begin{align}\label{skdfh}
{}&  \int_0^1 \int_{-1}^1 |g'(t+s)|^p\Big(\int_{r/|g'(t+s)|}^{1/\|g'\| _{L^p([s,1+s])}} x^p \,\nu(dx)\Big)\,ds\,dt\\ \label{skdfh-1}{}&\qquad  
   \leq K\Big(\int_{-1}^1 \|g'\|_{L^p([s,s+1])}^{\theta} \,ds+ \int_{-1}^1
    \|g'\|_{L^{\beta+\epsilon}([s,1+s])}^{\beta+\epsilon}\,ds\Big)<\infty
 \end{align}
 since $(\alpha-1)(\beta+\epsilon)>-1$. Thus,  \eqref{khshjl} follows by \eqref{khsk}, \eqref{eq93e4u7}--\eqref{eq93e4u7-1} and \eqref{skdfh}--\eqref{skdfh-1}.  
 
For  $p=\theta$ the above proof  goes through  except for \eqref{khsk}, where we need the additional assumption (A-log).  This completes the proof. 
 \end{proof}

\begin{lem}\label{lem-f-65}
For all $\zeta\in W^{k,p}$ we have  as $n\to \infty$, 
\begin{equation}\label{con-g-1}
n^{-1+pk}V(\zeta,p;k)_n\to \int_0^1 |\zeta^{(k)}(s)|^p\,ds.
\end{equation}
\end{lem}

\begin{proof}
First we will assume that $\zeta\in C^{k+1}(\R)$ and afterwards we will prove the lemma by approximation. Successive applications of Taylor's theorem gives
\begin{equation}\label{est-g-sf}
\Delta^n_{i,k} \zeta= \zeta^{(k)}\Big(\frac{i-k}{n}\Big) \frac{1}{n^k} +a_{i,n}, \qquad n\in \N, \ k\leq i\leq n
\end{equation}
where $a_{i,n}\in \R$ satisfies 
\begin{equation}
|a_{i,n}| \leq K n^{-k-1}, \qquad n\in \N, \ k\leq i\leq n. 
\end{equation}
By Minkowski's inequality, 
\begin{align}
 {}& \Big|\Big(n^{kp-1} V(\zeta)_n\Big)^{1/p} - \Big(n^{kp-1}
  \sum_{j=k}^n \Big|\zeta^{(k)}\Big(\frac{i-k}{n}\Big) \frac{1}{n^k}\Big|^p\Big)^{1/p}\Big|\\ {}& \qquad  \leq 
\Big(n^{pk-1} \sum_{j=k}^n |a_{i,n}|^p\Big)^{1/p}
\leq K n^{-1-1/p}\to 0. 
\end{align}
By continuity of $\zeta^{(k)}$ we have 
\begin{equation}
n^{kp-1} \sum_{i=k}^n \Big|\zeta^{(k)}\Big(\frac{i-k}{n}\Big) \frac{1}{n^k}\Big|^p \to \int_0^1 |\zeta^{(k)}(s)|^p\,ds
\end{equation}
as $n\to\infty$,  which   shows  \eqref{con-g-1}.

The statement of the lemma  for a general  $\zeta \in W^{k,p}$ follows  by approximating $\zeta$ through a sequence of $C^{k+1}(\R)$-functions 
and Minkowski's inequality.  This completes the proof. 
\end{proof}

\begin{proof}[Proof of Theorem~\ref{maintheorem}(iii)]
The Lemmas~\ref{abs-cont-sdf} and \ref{lem-f-65} yield Theorem~\ref{maintheorem}(iii). 
\end{proof}

\section{Proof of Theorem \ref{sec-order}} \label{sec5}
\setcounter{equation}{0}
\renewcommand{\theequation}{\thesection.\arabic{equation}}

We recall the definition of $m_p$ in Theorem \ref{maintheorem}(ii) and 
 of $g_{i,n}$ introduced  at \eqref{def-g-i-n}.
Throughout this section the driving L\'evy motion $L$
is a symmetric $\beta$-stable L\'evy process with scale parameter $1$.

\subsection{Proof of Theorem~\ref{sec-order}(i)}  \label{sec5.1}
Throughout the following subsection we assume  that the conditions  of 
Theorem~\ref{sec-order}(i) hold,  in particular,  $k=1$. For all $n\geq 1$ and $r\geq 0$ set
\begin{align}\label{eq:23412343}
\phi_{r}^n(s)={}& Dg_n(r-s) = n^{\alpha} \Big( g\big(\frac{r-s}{n}\big)-g\big(\frac{r-1-s}{n}\big)\Big),\\ 
 Y_{r}^n= {}& 
 \int_{-\infty}^r \phi_{r}^n(s)\,dL_s,\qquad 
V^n_r= | Y_r^n |^p - \E[ | Y^n_r |^p].
\end{align}
Due to self-similarity of $L$ of order $1/\beta$ we have for all $n\in \N$ that 
\begin{align} \label{statdec}
n^{1-\frac{1}{(1-\alpha)\beta}}\Big(n^{-1+p(\alpha + 1/\beta)}V(p;1)_n- m_p\Big) \eqschw S_n +  r_n
\end{align}
where 
\begin{equation}\label{def-S-n-23}
 S_n=  n^{1/(\alpha-1)\beta}\sum_{r=1}^n V^n_r \qquad \text{and}\qquad r_n =  n^{1-\frac{1}{(1-\alpha)\beta}} ( \E[| Y^n_1 |^p] -m_p).
\end{equation}
We will show Theorem~\ref{sec-order}(i) by showing that $r_n \to 0$ and  $S_n \schw S$, where $S$ is the limit introduced in Theorem \ref{sec-order}(i). The convergence   $S_n \schw S$  follows by the following Steps~1--3, whereas $r_n\to 0$ follows by    Step~4.

The following estimates will be useful for us. Let $W$ denote  a symmetric $\beta$-stable random variable with scale parameter $\rho\in (0,\infty)$ and set 
\begin{equation}\label{def-H-rho}
\Phi_\rho(x)= \E[| W+x|^p]- \E[ | W |^p],\qquad  x\in \R. 
\end{equation}
Using the representation \eqref{xp} it follows that 
\begin{equation}\label{rep-H-est-2}
\Phi_\rho(x)= a_p^{-1} \int_\R \big(1-\cos(ux) \big) e^{-\rho^\beta | u |^\beta}|u|^{-1-p}\,du. 
\end{equation}
Let $\epsilon>0$ be a fixed strictly positive number. 
From \eqref{rep-H-est-2}, we deduce that $H_\rho$ is two times continuous   differentiable, and  
\begin{align}\label{est-H''}
| \Phi_\rho''(x) |= a_p^{-1}\Big| \int_\R \cos(ux) |u|^{1-p}e^{-\rho^\beta | u |^\beta}\,du\Big| 
 \leq a_p^{-1} \int_\R  |u|^{1-p}e^{-\rho^\beta | u |^\beta}\,du,  
\end{align}
which implies that there exists a finite constant $K_\epsilon$ such that for all $\rho\geq \epsilon$ and all $x\in \R$ 
\begin{equation}\label{est-H-28217}
 |\Phi_\rho''(x)|\leq K_\epsilon.
\end{equation}
By  \eqref{rep-H-est-2} we also deduce the following estimate by several applications of the mean-value theorem 
\begin{equation}\label{est-H-inf-2}
|\Phi_\rho(x)-\Phi_\rho(y)| \leq K_\epsilon \Big(   \big(|x|\wedge 1+|y|\wedge 1\big)|x-y| \1_{\{|x-y|\leq 1\}} + |x-y|^p\1_{\{|x-y|>1\}}\Big) 
\end{equation} 
which holds for all $\rho\geq \epsilon$ and all $x,y\in\R$. Eq.~\eqref{est-H-inf-2} used on $y=0$ yields that 
\begin{align} \label{hestimate}
|\Phi_\rho (x)|\leq K_\epsilon (|x|^{p}\wedge |x|^{2} ),
\end{align}
which, in particular,  implies that 
\begin{equation}\label{est-H-beta}
 |  \Phi_\rho(x)|\leq K_\epsilon |x|^l\qquad \text{for all }l\in (p,\beta).
\end{equation}
Moreover, for all $r\in (p,2]$ and   $\rho_1, \rho_2\geq \epsilon$  we  deduce by
 \eqref{rep-H-est-2}  that 
\begin{equation}\label{est-H-rho-1-2}
 | \Phi_{\rho_1}(x)- \Phi_{\rho_2}(x) |\leq K_\epsilon | \rho_1^\beta - \rho_2^\beta | |x|^{r} \qquad \text{for all }x\in \R. 
\end{equation} 
We will also need the following estimate:
\begin{lem}\label{est-H-inf}
For all   $\kappa,\tau\in L^\beta([0,1])$ with   $\| \kappa \|_{L^\beta([0,1])} , \| \tau \|_{L^\beta([0,1])}\leq 1$  
set $U=\int_0^1 \kappa(s)\,dL_s$ and $V=\int_0^1 \tau(s)\,dL_s$. Moreover, let   $\Phi_\rho$ be given by \eqref{def-H-rho}.   For all  $r\in (1,\beta)$ and  $\epsilon>0$ we  have 
\begin{align}\label{est-we-want-3}
{}& \| \Phi_\rho(U)-\Phi_\rho(V) \|_{L^r}\\ {}& \qquad \leq  K_\rho\Big\{ \Big(  \| \kappa\|_{L^\beta([0,1])}^{\beta/r-1-\epsilon}+\| \tau  \|_{L^\beta([0,1])}^{\beta/r-1-\epsilon}  \Big)\| \kappa-\tau \|_{L^\beta([0,1])}       +   \| \kappa-\tau \|_{L^\beta([0,1])} ^{\beta/r}\Big\}.
\end{align}
\end{lem}

\begin{proof}
Let $q\in (r,\beta)$.
Using \eqref{est-H-inf-2}, Minkowski inequality and H\"older inequality on $q$ and $q':=rq/(q-r)$, which satisfies the equality $1/r=1/q+1/q'$, we obtain  
\begin{align}\label{eq:7238}
{}& \| \Phi_\rho(U)- \Phi_\rho(V)\|_{L^r} \\ \label{eq:7238-2}  \leq {}&    K_{\rho} \Big( \Big\{\| |U |\wedge 1\|_{L^{q'}} + \| |V|\wedge 1 ||_{L^{q'}} \Big\}\| U-V \|_{L^q} 
+ \| | U-V |^p  \1_{\{|  U-V |\geq 1\}}\|_{L^r}\Big) .
\end{align}
Since $q-r<1<r$ we have  $q'>q$, and hence 
\begin{equation}
 \| |V|\wedge 1 \|_{L^{q'}} \leq \E[ |V|^{q}]^{\frac{q-r}{q r}}\leq K  \| \tau \|_{L^\beta([0,1])}^{q/r-1} .
\end{equation}
A similar estimate holds for $U$. Let $W$ be a standard symmetric $\beta$-stable random variable. For all $b\in (0,1]$ and  all $v<\beta$ we have  the estimate 
\begin{equation}\label{eq:836}
 \E[ | b W |^v \1_{\{| b   W |\geq 1\}}]\leq K b^{\beta},
\end{equation}
which follows from the fact that $W$ has a density $\eta$ satisfying       
$\eta(x)\leq	 K(1+|x|)^{-1-\beta}$, $x\in \R$, see 
 e.g.\ \cite[Theorem~1.1]{Wat}, and hence by substitution 
\begin{align}
 \E[ | b W |^v\1_{\{| b W |\geq 1\}}] = {}&\int_\R |b x|^v \1_{\{ |bx|\geq 1\}}\,\eta(x)\,dx
\\ \leq {}& K b^{-1} \int_\R |x|^v \1_{\{|x|\geq 1\}} |b^{-1} x|^{-1-\beta}\,dx = K b^\beta. 
\end{align}
By applying  \eqref{eq:836} with $v= p r$ on the last term in \eqref{eq:7238-2} we obtain the desired estimate \eqref{est-we-want-3}. 
\end{proof}

Recall the definition of  $\phi^n_j$ in \eqref{eq:23412343}, and set  
\begin{equation}\label{def-rho}
 \phi_j(u)=(j-u)^\alpha_+- (j-u-1)^\alpha_+.
\end{equation}
From Lemma \ref{helplem} we obtain the  estimates 
\begin{equation}\label{eq:72138}
\| \phi^n_j\|_{L^\beta([0,1])}\leq K j^{\alpha-1},\qquad \| \phi_j \|_{L^\beta([0,1])}\leq  K j^{\alpha-1},
\end{equation}
which will be used 
repeatedly throughout the proof. 
Set, moreover, 
\begin{align}
\rho^n_j ={}&  \| \phi^n_1 \|_{L^\beta(\R\setminus [-j,-j+1])},  \qquad  \rho_0=\| \phi_1 \|_{L^\beta(\R)},\qquad 
U_{r,m}^n=  \int_{m}^{m+1} \phi_{r}^n(u)\,dL_u,
\end{align}
and for all $r\in \R$ set  
\begin{equation}
  \g_r=\sigma(L_s-L_u: s,u\leq r)\qquad\text{and}\qquad  \g_{r}^1=\sigma(L_s-L_u: r\leq s,u\leq r+1 ).
\end{equation}
We note  that $(\g^1_r)_{r\geq 0}$ is not a filtration. 
We are now ready to prove the first approximation of $(S_n)_{n\geq 1}$ in \eqref{def-S-n-23}.

\begin{proof}[Step~1]  Set 
\begin{equation}\label{eq:def-Z} 
S^{\prime}_n:=  n^{1/(\alpha-1)\beta} \sum_{r=1}^n Z_{r}^n, \qquad Z_{r}^n:= \sum_{j=1}^{\infty} \E[V^n_r |\g_{r-j}^1].
\end{equation}
 In the following we will show that 
\begin{equation}\label{eq:346}
 S_n - S^{\prime}_n \toop 0.
\end{equation}
Before showing \eqref{eq:346} let us show that the infinite series $Z^n_r$, defined in \eqref{eq:def-Z},  converges absolutely a.s. To this end fix $n$ and $r$ and  set $\theta_j :=  \E[V^n_r |\g_{r-j}^1]$, $j\geq 1$.  Notice  that  $\theta_j  = \Phi_{\rho^n_j} (U^n_{r,r-j}) - \E[ \Phi_{\rho^n_j} (U^n_{r,r-j})]$. Since  $\rho^n_j\to  \| \phi^n_1 \|_{L^\beta(\R)}$ 
as $j\to \infty$, we have for all $n\geq 1$ large enough that there exists  $j_0=j_0(n)$ such that  $\{\rho_{j}^n:j\geq j_0\}$  is  bounded away from zero, which combined with  \eqref{est-H-beta} and \eqref{eq:72138} implies that    for all $l\in (p,\beta)$
 \begin{equation}\label{eq:78723}
\E[ | \theta_j | ] \leq 2 \E[ | \Phi_{\rho^n_j}(U^n_{r,r-j}) | ] \leq K \E[ | U^n_{r,r-j} |^l ] \leq K j^{(\alpha-1) l }.
\end{equation}
Since the  right-hand side of \eqref{eq:78723} is summable in $j$ for all $l$ close enough to $\beta$ by the assumption $\alpha<1-1/\beta$, it follows that $\sum_{j=1}^\infty  \theta_j$ converge  absolutely a.s.\  

In the following we will show  \eqref{eq:346} by showing that $ V^n_r $ is suitably  close to $Z^n_r$, and to this aim  we will use the  following telescoping sum decomposition of $V^n_r$
\begin{align}\label{decom-H_n}
 V^n_r  = \sum_{j=1}^{\infty} \Big(\E[ V^n_r |\g_{r-j+1}] - \E[ V^n_r |\g_{r-j}]\Big),
\end{align}
which follows from the fact that  $\lim_{j\to \infty} \E[ V^n_r |\g_{r-j}]=\E[ V^n_r ]=0$ a.s. 
Such decompositions  go  back to \cite{hh97} in the discrete time moving average setting. 
By definition of $S_n'$ and \eqref{decom-H_n} we have 
\begin{align}\label{eq-S-S'-1}
 S_n - S^{\prime}_n ={}&  n^{1/(\alpha -1)\beta} \sum_{r=1}^{n} R_r^n,\qquad \qquad R^n_r:=\sum_{j=1}^\infty \zeta_{r,j}^n, \\ \label{eq-S-S'-2}
 \zeta_{r,j}^n:= {}& \E[ V^n_r |\g_{r-j+1}] - \E[ V^n_r |\g_{r-j}] - \E[ V^n_r | \g_{r-j}^1].
\end{align}
We now use a  telescoping sum decomposition of $\zeta^n_{r,j}$ 
\begin{equation}\label{def-theta-23}
 \zeta_{r,j}^n = \sum_{l=j}^\infty \vartheta_{r,j,l}^n ,\qquad \quad \vartheta_{r,j,l}^n := \E[\zeta_{r,j}^n | \g_{r-j}^1\vee \g_{r-l}] 
 - \E[\zeta_{r,j}^n  | \g_{r-j}^1\vee \g_{r-l-1}],
\end{equation}
which follows from the fact that  $\lim_{l\to \infty} \E[ \zeta^n_{r,j}|\g^1_{r-j}\vee \g_{r-l}]=\E[ \zeta^n_{r,j}|\g^1_{r-j}]=0$ a.s. The next lemma gives a moment estimate for $\vartheta^n_{r,j,l}$.  

\begin{lem}\label{lem-est-theta}
Let  $\vartheta^n_{r,j,l}$ be defined in \eqref{def-theta-23} and $\gamma\in [1,\beta)$. Then there exists  $N\geq 1$ such that  $\E[| \vartheta_{r,j,l}^n|^\gamma]\leq K_\gamma j^{(\alpha-1)\gamma}   l^{(\alpha-1)\gamma}$ for all $n\geq N, r\leq n, j\geq 1$ and  $l\geq j$. 
\end{lem}

\begin{proof}
For fixed $n,j,l$, $\{\vartheta^n_{r,j,l}:r\geq 1\}$ is a  stationary sequence, and hence we may and do assume that $r=1$. Furthermore, we may and do assume that $l\geq j\vee 2$, since the case  $l=j=1$  can be covered by choosing a new  constant $K_\gamma$. 
By definition of $\vartheta^n_{1,j,l}$ we obtain the representation 
\begin{equation}\label{eq:62627}
\vartheta^n_{1,j,l} =  \E[ | Y^n_1 |^p  | \g^1_{1-j}\vee \g_{1-l}] - \E[ | Y^n_1 |^p  |\g_{1-l}] - \E[ | Y^n_1 |^p  | \g_{1-j}^1\vee 
\g_{-l}] + \E[ | Y^n_1 |^p  | \g_{-l}].
 \end{equation}
Set $\rho_{j,l}^n = \| \phi^n_1 \|_{L^\beta([1-l,1-j]\cup [2-j,1])}$. For large enough $N\geq 1$ there exists $\epsilon>0$ such that $\rho_{j,l}^n\geq \epsilon$ for all $n\geq N, j\geq 1, l\geq j\vee 2$ (we have  $\rho_{j,l}^n=0$ for $l=1$). Hence by \eqref{est-H-28217} there exists a finite constant $K$ such that 
\begin{equation}\label{eq:73698}
|\Phi''_{\rho^n_{j,l}}(x)| \leq K \qquad \text{for all $n\geq N,\, j\geq 1,\, l\geq j\vee 2, \, x\in \R$}. 
\end{equation}
Let 
\begin{equation}
A_{l}^n= \int_{-\infty}^{-l} \phi_{1}^n(s) \,dL_s,  \qquad \qquad B_{j}^n=\int_{1-j}^{2-j} \phi_{1}^n(s)\,dL_s, 
\end{equation}
 and    $(\tilde B_{l+1}^n, \tilde B_{j}^n)$ denote a random vector which is independent of $L$ and which equals $ (B_{l+1}^n, B_{j}^n)$ in law. Let  moreover $\tilde \E$ denote expectation with respect to $(\tilde B_{l+1}^n, \tilde B_{j}^n)$ only. 
From \eqref{eq:62627}  we deduce that 
\begin{align}
\vartheta_{1,j,l}^n ={}&  \tilde \E\Big[  \Phi_{\rho^n_{j,l}}( A_{l}^n +B_{l+1}^n+B_{j}^n )- 
\Phi_{ \rho^n_{j,l}}(A_{l}^n+\tilde B_{l+1}^n+B_{j}^n)\\ {}&
\phantom{   \tilde \E\Big[ }   -\Phi_{\rho^n_{j,l}}(A_{l}^n+B_{l+1}^n+\tilde B_{j}^n)+ \Phi_{\rho^n_{j,l}}(A_{l}^n+\tilde B_{l+1}^n+\tilde B_{j}^n)\Big] \\ 
\label{eq:45} = {}& \tilde \E\Big[ \int^{B_{j}^n}_{\tilde B_{j}^n} \int^{B_{l+1}^n}_{\tilde B_{l+1}^n}  \Phi''_{\rho^n_{j,l}}(A_{l}^n+u_1+u_2)\,du_1\,du_2\Big],   
\end{align}
where $\int^x_y$ denotes $-\int^{y}_x$ if $x<y$. 
Hence by \eqref{eq:73698} we have 
\begin{align}
|  \vartheta_{r,j,l}^n |^\gamma \leq {}& K  \Big(  | B_{l+1}^n B_{j}^n |^\gamma +  |B_{j}^n|^\gamma  \E[ |B_{l+1}^n|^\gamma ] + \E[ |B_{j}^n|^\gamma] |B_{l+1}^n|^\gamma+\E[ | B_{j}^n|^\gamma ]\E[ | B_{l+1}^n |^\gamma]\Big), 
\end{align} 
and by independence of $B^n_{l+1}$ and $B^n_j$ for all $l\geq j$ we have 
\begin{equation}\label{eq:6128}
\E[  |  \vartheta_{r,j,l}^n |^\gamma ] \leq K \E[ | B_{j}^n|^\gamma ]\E[ | B_{l+1}^n |^\gamma] \leq K \| \phi^n_j \|_{L^\beta([0,1 ])}^\gamma \| \phi^n_{l+1} \|_{L^\beta([0,1 ])}^\gamma.
\end{equation}
Hence,   \eqref{eq:72138} completes the proof.  
\end{proof}

For all $m<-j$, the sequence $V_k:=\vartheta^n_{r,j,-k}$, for $k=m,\dots,-j$, is a martingale difference in the filtration 
$\tilde \g_k:= \g_{r-j}^1\vee \g_{r+k}$, $k=m,\dots,-j$. Hence by the  von Bahr--Esseen inequality  (see e.g.\ \cite[Lemma~4.2]{s02}) 
 we obtain for all $\gamma\in [1,\beta)$
\begin{equation}\label{est-345}
\E[| \zeta_{r,j}^n|^\gamma] \leq 2  \sum_{l=j}^\infty \E[| \vartheta_{r,j,l}^n |^\gamma].
\end{equation}
According to Lemma~\ref{lem-est-theta}, we have the estimate $\E[| \vartheta_{r,j,l}^n|^\gamma]\leq K j^{(\alpha-1)\gamma}   l^{(\alpha-1)\gamma}$,  which together with \eqref{est-345} implies that 
\begin{equation}\label{eq:77398}
\E[| \zeta_{r,j}^n|^\gamma]\leq K j^{(\alpha-1)\gamma}  \sum_{l=j}^\infty l^{(\alpha-1)\gamma  }
\leq K j^{2(\alpha-1)\gamma+1} .
\end{equation} 
Eq.~\eqref{eq:77398} used on   $\gamma = 1$ yields that  series 
 $R^n_r = \sum_{j=1}^\infty \zeta^n_{r,j}$ converges absolutely a.s.
Thus, by rearranging the terms in \eqref{eq-S-S'-1}--\eqref{eq-S-S'-2} using  the  substitution $s=r-j$, we have 
\[
S_n - S^{\prime}_n = n^{1/(\alpha -1)\beta } \sum_{s=-\infty}^{n-1} M_{s}^n \quad \text{with} \quad M_{s}^n:= \sum_{r= 1\vee (s+1)}^n \zeta_{r,r-s}^n.   
\]
Recalling the definition of $\zeta^n_{r,j}$ in \eqref{eq-S-S'-2}, we note that $\E[ \zeta^n_{r,r-s} |\g_s]=0$ for all $s$ and $r$,  showing that 
 that $\{M^n_s: s\in (-\infty,n)\cap \Z \}$ are  martingale differences.  Using again  von Bahr--Esseen inequality (\cite[Lemma~4.2]{s02}) we deduce that 
\begin{align*}
\E[|S_n - S^{\prime}_n|^{r'}]&\leq K n^{r'/(\alpha -1)\beta } \sum_{s<n} \E[|M_{n,s}|^{r'}] \\
&\leq K n^{r'/(\alpha -1)\beta } \sum_{s<n} \left( \sum_{r= 1\vee (s+1)}^n \E[|\zeta_{r,r-s}^n|^{r'}]^{1/r'}  \right)^{r'} \\
&\leq K n^{r'/(\alpha -1)\beta } \sum_{s<n} \left( \sum_{r= 1\vee (s+1)}^n (r-s)^{1/r' +2(\alpha -1)}  \right)^{r'} =: A_n.
\end{align*}

It now remains to prove that $A_n\rightarrow 0$, and to do this we distinguish two cases (recall that $1<r'<\beta$ close enough to $\beta$). 
Split $A_n= \sum_{-n<s<n} + \sum_{s<-n}=: A^{\prime}_n + A^{\prime \prime}_n$ and assume for the moment that $1/\beta  +2(\alpha -1)<-1$ 
holds. Since the inner sum is summable (for $r'$ close enough to $\beta$) we immediately see that $A^{\prime}_n \leq K n^{r'/(\alpha -1)\beta + 1}$.
Since $\beta >1$ and $\alpha \in (0, 1-1/\beta)$ we deduce that $
A^{\prime}_n \rightarrow 0$.  
On the other hand, a direct computation shows that
\[
A^{\prime \prime}_n \leq K n^{r'/(\alpha -1)\beta + 3 + 2r'(\alpha -1)},
\]
and since $\beta (\alpha -1)<-1$, we readily obtain that 
$
A^{\prime \prime}_n \rightarrow 0$.
Now, assume that  
$
1/\beta  +2(\alpha -1)\geq -1$.
Then
\[
A^{\prime}_n\leq K n^{r'/(\alpha -1)\beta + 2 +r' +2r'(\alpha-1) }.
\]
Since $1/\beta  +2(\alpha -1)\geq -1$, which is equivalent to $1  +2\beta (\alpha -1)\geq -\beta$, we obtain for $r'$ close enough to $\beta$
\[
A^{\prime}_n\leq K n^{1/(\alpha -1) + 1}\rightarrow 0.
\] 
The convergence $A^{\prime \prime}_n\rightarrow 0$ is shown as above. 
\end{proof}
Next we will show that $S_n'$ can be approximated by a rescaled sum of certain i.i.d.\ random variables $\tilde Z_r$, $r\geq 1$. 

\begin{proof}[Step~2]
Set  $W_{r,j}:= \int_{r}^{r+1} \phi_{j}(u)\,dL_u$ and  
\begin{equation}
\tilde S_n:=n^{1/(\alpha-1)\beta} \sum_{r=1}^n \tilde Z_r, \qquad \tilde{Z}_r :=
 \sum_{j=1}^{\infty} \Big\{\Phi_{\rho_0} ( W_{j+r,r})-\E[\Phi_{\rho_0} (W_{j+r,r})]\Big\},
\end{equation}
where $\rho_0$ has been defined at \eqref{def-rho}.
We will show that 
\begin{equation}
S_n' - \tilde S_n \toop 0.
\end{equation}
The series $\tilde Z_r$ converges absolutely a.s.\ according to the same arguments as used in Step~1.  
 Recall from Step~1 that  $\E[ V^n_r |\g_{r-j}^1]=  \Phi_{\rho^n_j}(U_{r,r-j}^n)- \E[ \Phi_{\rho^n_j}(U_{r,r-j}^n)]$
 and the series $S^{\prime}_n$ converges absolutely a.s.  
By rearranging the  terms in  $S^{\prime}_n$  by the substitution $s=r-j$, we obtain the decomposition 
  $S^{\prime }_n - S^{\prime \prime}_n
= R_n^1 - R_n^2 + R_n^3$, where
\begin{align*}
{}& R_n^1 :=  n^{1/(\alpha -1)\beta} \sum_{s<0} \left( \sum_{j=1-s}^{n-s} 
\Big\{\Phi_{ \rho^n_j} ( U^n_{j+s,s}) - \E[\Phi_{\rho^n_j} ( U^n_{j+s,s}) ] \Big\}\right), \\
{}& R_n^2 :=  n^{1/(\alpha -1)\beta} \sum_{s=0}^{n-1} \left( \sum_{j>n-s} 
\Big\{ \Phi_{\rho_0 }(W_{j+s,s})- \E[ \Phi_{\rho_0 }(W_{j+s,s})]\Big\}\right), \\
{}& R_n^3 :=   \\   {}& n^{1/(\alpha -1)\beta}\sum_{s=0}^{n-1} \left( \sum_{j=1}^{n-s} \Big\{
\Phi_{\rho^n_j} ( U^n_{j+s,s}) - \E[\Phi_{\rho^n_j} ( U^n_{j+s,s}) ]- \big( \Phi_{\rho_0 }(W_{j+s,s})- \E[ \Phi_{\rho_0 }(W_{j+s,s})] \big)\Big\} \right). 
\end{align*}
Now, we will show that all terms converge to 0 in probability. We start with the term $R_n^2$.  
For any $l\in (p,\beta)$ with 
 $l(\alpha -1)<-1$, we have deduce by  \eqref{est-H-beta} and \eqref{eq:72138}
 \begin{align*}
{}& \E[|R_n^2|] \leq 2 n^{1/(\alpha -1)\beta} \sum_{s=0}^{n-1} \sum_{j>n-s} 
\E\left[ \left| \Phi_{\rho_0 }(   W_{j+s,s} )\right|\right] 
\\
\leq {}&  K n^{1/(\alpha -1)\beta} \sum_{s=0}^{n-1}   \sum_{j>n-s} j^{l(\alpha -1)}  
\leq  K n^{1/(\alpha -1)\beta} \sum_{s=0}^{n-1}  s^{l(\alpha -1)+1}\leq K n^{1/(\alpha -1)\beta + l(\alpha -1)+2}.
\end{align*}
Choosing  $l$ arbitrary close to $\beta$, and taking into account  that
\[
2+ \beta (\alpha -1)< 1/(1-\alpha)\beta, 
\]
where the latter comes from the fact that $2-x < 1/x$ for any $x>1$, we conclude that $R_n^2 \toop 0$. 

In a similarly  way we   prove   $R_n^1 \toop 0$ in the following.  By our assumptions we may choose  $N\geq 1$ such that   $\{\rho^n_j: n\geq N, j\geq 1\}$ is bounded away from zero. For any $l\in (p,\beta)$ with 
 $l(\alpha -1)<-1$ and $n\geq N$ we have by \eqref{est-H-beta} and \eqref{eq:72138} that 
\begin{align}
  \E[ | R^1_n |] \leq {}&  2n^{1/(\alpha-1)\beta} \sum_{s<0} \sum_{j=1-s}^{n-s}\E[ | \Phi_{\rho^n_j} ( U^n_{j+s,s})| ]
  \leq  K n^{1/(\alpha-1)\beta} \sum_{s<0} 
 \sum_{j=1-s}^{n-s} j^{(\alpha-1) l }  \\ \leq {}&  K  n^{1/(\alpha-1)\beta} \sum_{s<0} 
\Big\{(1-s)^{(\alpha-1)l +1}  - (n-s)^{(\alpha-1)l +1} \Big\} \\ \label{eq:2272}
\leq {}& K n^{1/(\alpha-1)\beta} \sum_{u=1}^n u^{(\alpha-1)l +1} \leq K 
n^{1/(\alpha-1)\beta+ (\alpha-1)l +2}. 
\end{align} 
The estimate \eqref{eq:2272}  implies  $R^1_n\toop 0$ as above.  

%

Next we will show that $R_n^3 \toop 0$. 
We start with  the following simple estimate 
\begin{align}
\E[|R_n^3|]\leq {}&  2 n^{1/(\alpha-1)\beta}\sum_{s=0}^{n-1} \sum_{j=1}^{n-s} 
\E\Big[\Big|  \Phi_{\rho^n_j} ( U^n_{j+s,s})- \Phi_{\rho_0 }(V_{j+s,s})\Big|\Big]
 \\ \label{eq:4729} \leq {}&  2 n^{1/(\alpha-1)\beta+1}\ \sum_{j=1}^{n} \E\Big[\Big|  \Phi_{\rho^n_j} ( U^n_{j,0})- \Phi_{\rho_0 }(V_{j,0})\Big|\Big].
\end{align}
Next we decompose 
\begin{align}\label{decomp-H-widehat-2}
\Phi_{\rho^n_j} ( U^n_{j,0})- \Phi_{\rho_0 }(V_{j,0}) = C_{j}^n+Q^n_j
\end{align}
where 
\begin{equation}
 Q^n_j= \Phi_{\rho^n_j}( U_{j,0}^n)- \Phi_{\rho_0}(U_{j,0}^n)\qquad\text{and}\qquad  C^n_j=\Phi_{\rho_0}(U_{j,0}^n)-
 \Phi_{\rho_0}(V_{j,0}).
\end{equation}
We have that 
\begin{align}
U^n_{j,0}-W_{j,0}= \int_0^1 \zeta^n_j(u)\,dL_u\qquad\text{where}\qquad  \zeta^n_j(u):=  \phi_{j}^n(u)- \phi_j(u).
\end{align}
In the following we will prove and  use the estimate  
\begin{equation}\label{eq:72138-2}
 \| \zeta^n_j\|_{L^\beta [0,1]}\leq K n^{-1} j^\alpha.
\end{equation}
Recall that  $g(s)=s^{\alpha} f(s)$ for $s\geq 0$.  
The estimate \eqref{eq:72138-2}  follows by the decompose
\begin{equation}\label{eq:723161}
\zeta_{j}^n = \hat\zeta_{j}^n + \bar \zeta^n_j
\end{equation}
with 
\begin{align*}
\hat \zeta_j^n(u) &=  \big[f((j-u)/n) -f(0)\big]\phi_j(u), \\
\bar \zeta^n_j(u) &=  \big[f((j-u)/n) -f((j-u-1)/n)\big] (j-u-1)_+^\alpha,
\end{align*}
 the triangle inequality and  continuous right differentiability of $f$ at zero.  Fix  $\epsilon>0$. 
 By Lemma~\ref{est-H-inf} used on $r=1$,   \eqref{eq:72138} and \eqref{eq:72138-2} 
 we have
\begin{align}\label{eq:836a}
 \sum_{j=1}^n \E[ |C^n_j|]\leq {}& K \sum_{j=1}^n \big(  j^{(\alpha-1)(\beta-1-\epsilon)} n^{-1}j^{\alpha} +n^{-\beta } j^{\alpha \beta} \big)\leq K n^{(\alpha-1)(\beta-\epsilon)+1}\to 0\quad 
\end{align}
for $\epsilon$ small enough. 
By substitution, Lemma~\ref{helplem1}(i) and \eqref{eq:72138}, 
\begin{align}
{}&  \Big | |\rho_{j}^n|^\beta-   |\rho_0|^\beta\Big|  \leq  \Big| \| \phi^n_1 \|_{L^\beta(\R)}- \| \phi_1 \|_{L^\beta(\R)} \Big| 
+ \| \phi^n_1 \|_{L^\beta([-j,-j+1])}
 \\  {}& \leq K\Big(  n^{\alpha \beta+1} \Big| \| g_{0,n}\|_{L^\beta(\R)}^\beta -  \| h_{0,n}\|_{L^\beta(\R)}^\beta  \Big| + j^{\alpha-1}\Big)\leq K\Big( n^{(\alpha-1)\beta+1} +j^{\alpha-1}\Big).\label{eq:63734}
\end{align} 
For any $r\in (p,\beta)$ such that $(\alpha-1)r <-1$ we have  by \eqref{est-H-rho-1-2}, \eqref{eq:72138} and \eqref{eq:63734} that 
\begin{align}
\E[  | Q^n_j | ] = {}&  \E[  | \Phi_{\rho^n_j}( U_{j,0}^n)-\Phi_{ \rho_0}(U_{j,0}^n)| ] 
\leq K  \big| |\rho^n_j |^\beta-  | \rho |^\beta \big|  \E[ | U_{j,0}^n |^r   ]  
\\ \leq {}& K  \big| |\rho^n_j |^\beta-  | \rho |^\beta \big|  \| \phi^n_j \|_{L^\beta([0,1])}^r \leq K \Big(n^{(\alpha-1)\beta+1} j^{(\alpha-1)r }+j^{(\alpha-1)(1+r)}\Big),
 \end{align}
which implies that 
\begin{equation}\label{est-Q^n_r}
  \sum_{j=1}^n \E[ |Q^n_j|] \leq K\Big( n^{(\alpha-1)\beta+1} + n^{(\alpha-1) (1+r)+1}\Big)\leq K n^{( \alpha-1)\beta +1} .
\end{equation}
The equations   \eqref{eq:4729}, \eqref{decomp-H-widehat-2},  \eqref{eq:836a} and  \eqref{est-Q^n_r}  show that 
\begin{equation}
\E[ |R^3_n| ]\leq K n^{1/(\alpha-1)\beta+(\alpha-1)\beta+2},
\end{equation}
which implies that $R^3_n\toop 0$  as above, and the proof is complete. 
\end{proof}

We are now ready to complete the proof of $S_n\schw S$.

%

\begin{proof}[Step~3]
We will show that $\tilde S_n\schw S$. 
Since the random variables $(\tilde Z_r)_{r\geq 0}$ are independent and identically distributed with mean zero, 
it is enough to show that  
\begin{equation}\label{eq:93673}
\lim_{x\to\infty}x^{(1-\alpha)\beta}\P(\tilde Z_0>x)= \gamma \qquad \text{and}\qquad \lim_{x\to\infty}x^{(1-\alpha)\beta} \P(\tilde Z_0<-x)=0
\end{equation}
cf.\  
 \cite[Theorem~1.8.1]{SamTaq}. The 
constant $\gamma$ is defined in \eqref{eq:24245} below. 
To show \eqref{eq:93673} 
let us define the function $\overline{\Phi}:\R\to \R_+$ via
\[
\overline{\Phi}(x):= \sum_{j=1}^{\infty} \Phi_{\rho_0} (a_j x) \qquad \text{where} \qquad a_j:=j^{\alpha } - (j-1)^{\alpha}. 
\]
Note that \eqref{rep-H-est-2} implies that $\Phi_{\rho_0} (x)\geq 0$ and hence $\overline \Phi$ is positive. Moreover, by \eqref{est-H-beta}  and for $l\in (p,\beta)$ with $(1-\alpha)l <-1$ we have 
\begin{equation}\label{eq:232}
|\overline{\Phi}(x)|\leq K |x|^l \sum_{j=1}^{\infty} a_j^\beta \leq K |x|^l\sum_{j=1}^{\infty} j^{\beta(\alpha -1)}<\infty, 
\end{equation}
 which shows that $\overline \Phi$ is finite. 
 Eq.~\eqref{eq:232} shows moreover that $\E[ \overline \Phi(L_1)]<\infty$, and hence we can define a random variable $Q_0$ via 
 \begin{equation}
Q_0 = \overline \Phi(L_1)- \E[ \overline \Phi(L_1)] = \sum_{j=1}^{\infty} \big(\Phi_{\rho_0} (a_j L_1 )- \E[ \Phi_{ \rho_0}(a_j L_1)]\big),
 \end{equation}
 where the last sum converge absolutely a.s.
Since 
 $Q_0\geq - \E[ \overline \Phi(L_1)] $, we have that 
 \begin{equation}\label{eq:836b}
 \lim_{x\to\infty}x^{(1-\alpha)\beta} \P(Q_0<-x)=0. 
 \end{equation}
By the substitution $t=(x/u)^{1/(1-\alpha)}$ we have that 
\begin{align}
x^{1/(\alpha-1)} \overline \Phi(x)= {}& x^{1/(\alpha-1)} \int_0^\infty \Phi_{ \rho_0}(a_{1+[t]} x)\,dt
\\ = {}& (1-\alpha)^{-1}  \int_0^\infty \Phi_{\rho_0}(a_{1+[( x  / u)^{1/(1-\alpha)}]} x)u^{-1+1/(\alpha-1)}\,du \\ \label{eq:7268712}
\to {}& (1-\alpha)^{-1} \int_0^\infty \Phi_{\rho_0}(\alpha u)u^{-1+1/(\alpha-1)}\,du=\kappa\qquad \text{as }x\to \infty,
\end{align}
where we have used Lebesgue's dominated convergence theorem and the estimate  \eqref{hestimate} on $\Phi_{\rho_0}$, and the constant $\kappa$ coincides with the definition in Remark 
\ref{rem-const}. The connection between the tail behaviour 
of a symmetric $\rho$-stable random variable $S_{\rho}$, $\rho\in (1,2)$,  and its  scale parameter $\bar\sigma$ is given via 
\[
\mathbb{P}(S_{\rho}>x) \sim  \tau_{\rho} \bar\sigma^{\rho} x^{-\rho} /2\qquad \text{as } x\rightarrow \infty, 
\]  
where the function $\tau_{\rho}$ has been defined in \eqref{def-tau-rho}
(see e.g.\ \cite[Eq. (1.2.10)]{SamTaq}). Hence,  $\mathbb P (|L_1|>x)\sim \tau_{\beta}  x^{-\beta}$ as $x\rightarrow \infty$, and we readily deduce by \eqref{eq:7268712} that 
\begin{equation}\label{eq:24245}
\mathbb P(Q_0>x)\sim \gamma x^{(1-\alpha)\beta} \quad \text{with} \quad \gamma = 
\tau_{\beta}  \kappa ^{(1-\alpha )\beta }.
\end{equation}

Next we will show that 
for some $r>(1-\alpha)\beta$ we have 
\begin{equation}\label{eq:434}
\mathbb P(|\tilde Z_0 - Q_0|>x)\leq Kx^{-r} \qquad \text{for  all } x\geq 1,
\end{equation}
which implies  \eqref{eq:93673}, cf.\ \eqref{eq:836b} and \eqref{eq:24245}.
To show \eqref{eq:434}  it is sufficient to find  $r>(1-\alpha)\beta$ such that 
\[
\E[|\tilde Z_0 - Q_0|^r]<\infty  
\]
by Markov's inequality. Furthermore,   by Minkowski inequality and the definitions of $Q_0$ and $\tilde Z_0$  it sufficers to show that 
\begin{equation}
 \label{eq:6739}
\sum_{j=1}^{\infty} \| \Phi_{\rho_0}(W_{j,0} )
- \Phi_{\rho_0}(  a_j L_1)\|_{L^r}<\infty  
\end{equation}
(recall  that $r\geq 1$).
We choose an $r$ satisfying $(1-\alpha)\beta<r< \beta$, which is always possible since we have $\alpha\in (0,1)$ under our assumptions. 
%
We note that $\phi_j(0)=a_j$, and hence obtain the estimates
\begin{align} \label{estimate-262}
\| \phi_j-a_j\|_{L^{\beta}([0, 1])}\leq K j^{\alpha -2},\qquad j\geq 1. 
\end{align}
For all $\epsilon>0$ we have by   Lemma~\ref{est-H-inf}, \eqref{eq:72138} and  \eqref{estimate-262} that 
\begin{equation}
 \| \Phi_{\rho_0}( W_{j,0} )
- \Phi_{\rho_0}(  a_j L_1)\|_{L^r}\leq K\Big( j^{ (\alpha-1)(\beta/r-\epsilon)-1}+ j^{(\alpha-2)\beta/r}\Big)
\end{equation}
which shows \eqref{eq:6739}, by choosing  $\epsilon= \beta/(2r)$ and noticing that $(\alpha-2)\beta/r<\alpha-2<-1$.  
\end{proof}

In the last step we will show that $r_n = n^{1-\frac{1}{(1-\alpha)\beta}} ( \E[| Y^n_1 |^p] -m_p)\to 0$. 
\begin{proof}[Step~4]
 Let $g_{0,n}$ and $h_{0,n}$ be  defined as in \eqref{def-g-i-n}, and 
  $\eta_p$ be the absolute $p$-moment of a standard symmetric $\beta$-stable random variable. By  Lemma~\ref{helplem1}(i)  
 \begin{align}
  |r_n| =  {}& \eta_p n^{1-\frac{1}{(1-\alpha)\beta}}n^{(\alpha +1/\beta)p } \big| \|g_{0,n}\|_{L^{\beta}}^{\beta }
-\| h_{0,n}\|_{L^{\beta}}^{\beta }  \big| 
\\ \leq {}& K n^{1-\frac{1}{(1-\alpha)\beta}+ (\alpha +1/\beta)p-\beta}\leq K n^{2-\frac{1}{(1-\alpha)\beta}+ (\alpha -1)\beta}\label{eq:7370}
 \end{align}
where the last inequality follows from the fact that   $p<\beta$.  Eq.~\eqref{eq:7370} implies that  
$r_n\to 0,$ which completes the proof. 
 \end{proof}

\subsection{Proof of Theorem \ref{sec-order}(ii)} \label{sec5.2}
We use the following decomposition:
\begin{align} \label{statdec2}
\sqrt{n}\Big(n^{-1+p(\alpha + 1/\beta)}V(p;k)_n- m_p\Big) =  Y_n^{(1)} +  Y_n^{(2)}
\end{align}
with
\begin{equation}
Y_n^{(1)} := \sqrt{n}\Big(n^{-1+p(\alpha + 1/\beta)}V(p;k)_n- m_p^n\Big), \qquad 
Y_n^{(2)} := \sqrt{n} (m_p^n -m_p).
\end{equation} 
We will prove that $Y_n^{(2)} \rightarrow 0$ and 
\[
Y_n^{(1)} = \frac{1}{\sqrt{n}} \sum_{i=k}^n Y_i^n \schw \mathcal N(0,\eta^2), \quad \text{where}\quad Y_i^n = n^{p(\alpha +1/\beta)} \Big(
|\Delta_{i,k}^n X|^p - \E[|\Delta_{i,k}^n X|^p] \Big)
\]
and  $\eta^2$
is defined at \eqref{def-eta-23}. The proof of the latter relies on the short memory approximation of the random variables $Y_i^n$. 
Recalling the representation $\Delta_{i,k}^n X = \int_{\R} g_{i,n}(s) \,dL_s$, we introduce the random variables
\begin{align*}
\Delta_{i,k}^{n,m} X := \int_{\R} g_{i,n}(s) \1_{[(i-m)/n,(i+m)/n]}(s) \,dL_s.  
\end{align*}
We will proceed by showing the convergence in distribution
\begin{equation}\label{eq:6725}
Y_{n,m}^{(1)} = \frac{1}{\sqrt{n}} \sum_{i=k}^n Y_i^{n,m} \schw \mathcal N(0,\eta^2_m) \qquad \text{as } n\rightarrow \infty, 
 \end{equation}
where $Y_i^{n,m} = n^{p(\alpha +1/\beta)} \Big(
|\Delta_{i,k}^{n,m} X|^p - \E[|\Delta_{i,k}^{n,m} X|^p] \Big)$ and $\eta_m^2$ are positive constants.
Next, we show that $\eta_m^2\rightarrow \eta^2$ as $m\rightarrow \infty$. 
In the last step we prove that
\begin{equation}\label{eq:672323}
\lim_{m\rightarrow \infty} \limsup_{n\rightarrow \infty} \E[|Y_{n,m}^{(1)} -Y_{n}^{(1)}|^2] = 0.
\end{equation}
\emph{Step~(i).} To show \eqref{eq:6725} we note that for fixed $n,m\geq 1$, $\{Y_i^{n,m}: i=k,\dots,n\}$ are $m$-dependent random variables. 
By stationarity and $m$-dependence of $\{Y_i^{n,m}: i=k,\dots,n\}$ we have 
\begin{align}\label{eq:72727}
\text{var}(Y_{n,m}^{(1)}) = \frac{n-k}{n}\theta^m_n(0) + 2n^{-1} \sum_{i=1}^{m}(n-k-i) \theta^m_n(i)
\end{align}   
with $\theta^m_n(i)= \text{cov}(Y^{n,m}_k,Y_{k+i}^{n,m})$. 
Set 
\begin{equation}
V^m_i = \int_{\R} h_k(i-s) \1_{[(i-m)/n,(i+m)/n]}(s) \,dL_s\qquad \text{and}\qquad Y^m_i = | V^m_i|^p - \E[  | V^m_i|^p]. 
\end{equation}
By  the the line above  \eqref{con-L-1} we deduce that for all $d\geq 1$ and as $n\to\infty$, $(Y_i^{n,m})_{i=1}^d \schw (Y_i^m)_{i=1}^d$.
 For any $q>0$ with $pq<\beta$, the estimates of the proof of 
 Theorem~\ref{maintheorem}(ii) show that  
\begin{equation}\label{eq:92882}
 \mathbb E[|Y_i^{n,m}|^q] \leq K,
\end{equation}
and by the assumption $p<\beta/2$ we may choose $q>2$. Hence,  with $\theta^m(i):= \text{cov}(V^{m}_k,V_{k+i}^{m})$, it follows that $\theta^m_n(i)\to \theta^m(i)$ as $n\to\infty$, which by  \eqref{eq:72727} implies that 
\begin{equation}\label{eq:76726}
 \text{var}(Y_{n,m}^{(1)}) \to \theta^m_n(0) + 2\sum_{i=1}^{m} \theta^m(i)=:\eta^2_m\qquad \text{as }n\to \infty. 
\end{equation} 
By \eqref{eq:92882} and \eqref{eq:76726}, the convergence \eqref{eq:6725} follows by the main theorem in \cite{Berk} since the sequences $\{Y_i^{n,m}: i=k,\dots,n\}$ are  $m$-dependent for all $n\geq 1$.  
\qed

\noindent
\emph{Step~(ii).} To show that $\eta_m^2\to \eta^2$ as $m\to \infty$ we set  
\begin{equation}
V_i = \int_{\R} h_k(i-s)  \,dL_s,\qquad Y_i = | V_i|^p - \E[  | V_i|^p], \qquad \theta(i)= \textrm{cov}( Y_k, Y_{k+i}). 
\end{equation} By the definition of $\theta^m(i)$ and the continuity  of the stochastic integrals (cf.\ \cite{RajRos}) we have that  $\theta^m(i)\to \theta(i)$ as $m\to \infty$. Applying the formulas \eqref{xp} and  \eqref{charfun} we obtain the  expression
\begin{align}
\label{thetam} \theta^m(i) ={}&  a_p^{-2} \int_{\R^2} \frac{1}{|s_1s_2|^{1+p}} \psi_i^{m}(s_1,s_2)\, ds_1\, ds_2, \qquad \text{where}\\
\psi_i^{m}(s_1,s_2) ={}&  \exp \left(-  \int_{\R} |s_1 h_k^m(x) -s_2h_k^m(x+i)|^{\beta } dx \right) \nonumber  \\[1.5 ex]
{}&- \exp \left(-  \int_{\R} |s_1 h_k^m(x)|^{\beta} + |s_2h_k^m(x+i)|^{\beta} dx \right) \nonumber
\end{align}
and  $h_k^m (x) = h_k(x)\1_{[-m,m]}(x)$.  The functions $u= h_k^m$ satisfies the estimate \eqref{ufunction} for all $n\geq 1$ since 
$|h_k^m(x)| \leq K |x|^{\alpha}$ for $x \in [0,k+1]$ and $|h_k^m(x)| \leq K |x|^{\alpha-k}$ for $x>k+1$.  Since $u$ does not depend on $n\geq 1$ we obtain, by letting $n\to\infty$ in the estimate of Lemma~\ref{helplem2}(ii), that there exists some $r>1$ such that 
\begin{align} \label{eq}
|\theta^m(l)| \leq K l^{-r }\qquad \text{for all }m\geq 1, \, l\geq 0. 
\end{align}
Hence, by dominated convergence we have as $m\to\infty$, 
\[
\eta_m^2=\theta^m(0) + 2\sum_{i=1}^{m} \theta^m(i) \to
\theta(0) + 2\sum_{i=1}^{\infty}\theta(i)=\eta^2.
\] 
\qed

\noindent \emph{Step~(iii).} 
In this step we prove \eqref{eq:672323}. 
%
%
%
%
  By stationarity we have that 
\begin{align}\label{eq:232882}
\E[|Y_{n,m}^{(1)} -Y_{n}^{(1)}|^2] \leq 2 \sum_{i=0}^{n-k} \overline{\theta}^m_n(i)
\end{align}
where $\overline \theta^m_n(i)= \text{cov}(Y^{n,m}_k-Y^n_k,Y_{k+i}^{n,m}-Y_{k+i}^n)$. We introduce the notation
\[
f_{n}^m (x):= \1_{[-m,m]}(x) D^k g_n(x), \qquad f_{n} (x):=  D^k g_n(x), \qquad \overline{f}_{n}^m = f_{n}^m - f_{n}
\]
where the functions $g_n$ are defined at \eqref{def-g-i-n}. 
By \eqref{xp} and  \eqref{charfun} we have 
\begin{align}
\label{thetamn2} 
\overline{\theta}^m_n(i) = {}& a_p^{-2} \int_{\R^2} \frac{1}{|s_1s_2|^{1+p}} \overline{\psi}_i^{n,m}(s_1,s_2)\, ds_1 \,ds_2, \\
\overline{\psi}_i^{n,m}(s_1,s_2) ={}&  \exp \left(- \int_{\R} |s_1 \overline f_{n}^m(x) -s_2
\overline f_{n}^m(x+i)|^{\beta } \,dx \right) \nonumber  \\[1.5 ex]
{}&- \exp \left(- \int_{\R} |s_1 \overline f_{n}^m(x)|^{\beta} + |s_2 \overline f_{n}^m(x+i)|^{\beta} \,dx \right). \nonumber
\end{align}
By Lemma~\ref{helplem} it follows that  $u=\overline{f}_{n}^m$ satisfies \eqref{ufunction}. Since $u$ in addition  vanishes  on $[0,m]$,  
Lemma~\ref{helplem2}(iii) and \eqref{eq:232882}  implies that there  exists $r>0$ such that 
\begin{equation}
\limsup_{n\to\infty} \E[|Y_{n,m}^{(1)} -Y_{n}^{(1)}|^2] \leq K m^{-r} 
\end{equation}
which shows \eqref{eq:672323}. 
\qed
%
%
%
%
%
%
%
%
%
%
%
%
%
%
%
%

Steps~(i)--(iii) implies that  $Y_n^{(1)} \schw \mathcal N(0,\eta^2)$, and the  convergence $Y_n^{(2)} = \sqrt{n} (m_p^n -m_p)\rightarrow 0$ follows by Step~(iv) below.  

\noindent
\emph{Step~(iv).} Following the same arguments as in Step~4 of the previous subsection, we readily deduce that
\[
|Y_n^{(2)}|\leq K n^{\alpha \beta +3/2} \left |\|g_{0,n}\|_{L^{\beta}(\R)}^{\beta }
-\| h_{0,n}\|_{L^\beta(\R)}^{\beta } \right|. 
\]
Applying Lemma \ref{helplem1}(ii) immediately shows that $Y_n^{(2)}\rightarrow 0$ as $n\to \infty$. 
\qed

\section{Appendix}
\setcounter{equation}{0}
\renewcommand{\theequation}{\thesection.\arabic{equation}}
In this section we present some technical results that are used in various proof steps.
Recall the definition of the functions $g_{i,n}$ and $h_{i,n}$ in  \eqref{def-g-i-n} and \eqref{def-h-i-n}.
Recall that $f:\R_+\to \R$  is given by $f(x)=g(x)x^{-\alpha}$ for all $x>0$, and $f(0)=1$. We may and do extend $f$ to a function $f:\R\to\R$ (also denoted $f$) which is differentiable on $\R$.

\begin{lem} \label{helplem1}
\begin{itemize}
\item[(i)] Assume that the conditions of Theorem~\ref{sec-order}(i) hold. Then, 
\[
\left| \|g_{0,n}\|_{L^\beta(\R)}^{\beta} - \| h_{0,n}\|_{L^\beta(\R)}^{\beta} \right | \leq K n^{-\beta}. 
\]
\item[(ii)] Assume that the conditions of Theorem~\ref{sec-order}(ii) hold. Then 
\[
\left| \|g_{0,n}\|_{L^\beta(\R)}^{\beta} - \| h_{0,n}\|_{L^\beta(\R)}^{\beta} \right | \leq K n^{-\alpha \beta -2}. 
\]
\end{itemize}
\end{lem}

\begin{proof}
We recall that $g(x)=x^{\alpha}_+f(x)$ and $f(0)=1$. We start by proving part (ii) of the lemma.

Recall that $k\geq 2$ and $\alpha \in (0,k-2/\beta)$.
 By Lemma \ref{helplem} 
and condition $\alpha < k- 1/\beta$ it holds for all $n\geq 1$ that 
\begin{align} \label{est1}
A_n:=\int_{-\infty}^{-1} |h_{0,n} (x)|^{\beta}\, dx \leq K n^{-k\beta } \int_{1}^{\infty} x^{(\alpha -k)\beta} \,dx \leq K n^{-k\beta}.
\end{align}
The same estimate holds for the function $g_{0,n}$.
On the other hand, we have that 
\begin{align} \label{est2}
B_n:=\left| \int_{-\frac{2 k}{n}}^{0} |g_{0,n} (x)|^{\beta}\, dx - \int_{-\frac{2 k}{n}}^{0} | h_{0,n} (x)|^{\beta} \,dx \right|
\leq K n^{-\alpha \beta -2},
\end{align}
which follows by the estimate $| |x|^\beta - |y|^\beta|\leq K \max\{ |x|^{\beta-1},|y|^{\beta-1}\} |x-y|$ for all $x,y>0$, and that for all  $x\in [-\frac{2 k}{n}, 0]$ we have by differentiability of $f$ at zero that 
\begin{equation}
| g_{0,n}(x)- h_{0,n}(x)| = \Big| \sum_{j=0}^k (-1)^j \binom{k}{j} \big\{f(-j/n-x)
-1\big\}  (-j/n-x)_+^\alpha\Big|\leq K n^{-1} |h_{0,n}(x)| ,
\end{equation}
together with the estimate \eqref{lemest1} from Lemma~\ref{helplem} on $h_{0,n}$ and $g_{0,n}$. 
Recalling that $g(x)=x^{\alpha}_{+} f(x)$ and using $k$th order Taylor expansion of $f$ at $x$, we deduce the following identity
\begin{align*}
g_{0,n}(x)={}&  \sum_{j=0}^k (-1)^j \binom{k}{j} g\big(-j/n-x\big) \\
= {}& \sum_{j=0}^k (-1)^j \binom{k}{j} \big(-j/n-x\big)_{+}^{\alpha}
\left( \sum_{l=0}^{k-1} \frac{f^{(l)}(-x)}{l!} (-j/n)^l + \frac{f^{(k)}(\xi_{j,x})}{k!} (-j/n)^k \right) \\
={}&  \sum_{l=0}^{k-1} \frac{f^{(l)}(-x)}{l!} \left( \sum_{j=0}^k (-1)^j \binom{k}{j} (-j/n)^l \big(-j/n-x\big)_{+}^{\alpha} \right) \\
{}&+ \left( \sum_{j=0}^k  \frac{f^{(k)}(\xi_{j,x})}{k!} (-1)^j \binom{k}{j} (-j/n)^k \big(-j/n-x\big)_{+}^{\alpha} \right),  
\end{align*}
where $\xi_{j,x}$ is a certain intermediate point. Now, by rearranging terms we can find coefficients $\lambda_0,\cdots,\lambda_k$
and $\tilde{\lambda}_0,\cdots,\tilde{\lambda}_k$
(which are in fact bounded functions in $x$)
such that  
\begin{align*}
g_{0,n}(x)&= \sum_{l=0}^k \lambda_l(x) n^{-l} \left( \sum_{j=l}^k (-1)^j \binom{k}{j} j(j-1) \cdots (j-l+1) \big(-j/n-x\big)_{+}^{\alpha} \right) \\
&= \sum_{l=0}^k \tilde{\lambda}_l(x) n^{-l} \left( \sum_{j=l}^k (-1)^j \binom{k-l}{j-l}  \big(-j/n-x\big)_{+}^{\alpha} \right) 
=: \sum_{l=0}^k r_{l,n} (x). 
\end{align*}
At this stage we remark that the term $r_{l,n} (x)$ involves $(k-l)$th order differences of the function $(-x)_{+}^\alpha$ (at scale $n^{-1}$
and with a certain shift) and $\lambda_0(x)=\tilde{\lambda}_0(x)=f(-x)$. 
Now, observe that 
\begin{align*}
 C_n:={}& 
\int_{-1}^{-\frac{2 k}{n}} \left| |g_{0,n}(x)|^{\beta} - |h_{0,n}(x)|^{\beta} \right|  
dx \\
 \leq {}&  K \int_{-1}^{-\frac kn} \max\{|g_{0,n}(x)|^{\beta-1}, |h_{0,n}(x)|^{\beta-1}\} |g_{0,n}(x)- h_{0,n}(x)| \,dx.
\end{align*}
Since $r_{0,n}=f h_{0,n}$,  it holds  that 
\begin{equation}
|r_{0,n}(x) - h_{0,n}(x) | \leq K  n^{-k} x^{\alpha-k+1}.
\end{equation}
Recalling that $\alpha \in (0,k-2/\beta)$ and using the substitution $x=n^{-1}y$, we deduce that 
\begin{align} \label{est4}
\int_{-1}^{-\frac{2 k}{n}} \max\{|g_{0,n}(x)|^{\beta-1}, |h_{0,n}(x)|^{\beta-1}\}  |r_{0,n}(x)- h_{0,n}(x)|  
\,dx \leq K n^{-\alpha \beta - 2}.
\end{align}  
For $1\leq l\leq k$,  we readily obtain the approximation
\begin{align*} 
\int_{-1}^{-\frac{2 k}{n}} \max\{|g_{0,n}(x)|^{\beta-1}, | h_{0,n}(x)|^{\beta-1}\} |r_{l,n}(x)| \,dx 
\leq K n^{-l-1-\alpha \beta} \int_{k}^{n} y^{(\alpha -k)\beta  +l } \,dy 
\end{align*}
using again the substitution $x=n^{-1}y$. It holds that 
\begin{align} \label{est3}
 \int_{k}^{n} y^{(\alpha -k)\beta  +l } dy \leq 
\begin{cases}
K & (\alpha -k)\beta  +l < -1 \\
K \log(n)n^{(\alpha -k)\beta  +l+1} & (\alpha -k)\beta  +l \geq -1.
\end{cases}
\end{align}
By \eqref{est4} and \eqref{est3} we conclude that  
\begin{align*} 
C_n
\leq K \left( n^{-\alpha \beta - 2}+ \log(n)
n^{-k \beta} \right).
\end{align*}
Since $\alpha \beta +2<k\beta$ due to the assumption $\alpha \in (0,k-2/\beta )$, part (ii) readily follows from \eqref{est1} and \eqref{est2}.

Now, we proceed with part (i), which is in fact easier. Recall that $k=1$ and $\alpha \in (0,1-1/\beta)$, which implies that $\beta \in (1,2)$. 
As in \eqref{est1} and \eqref{est2}  it holds that
\begin{align} \label{est6}
A_n=
\int_{-\infty}^{-1} |h_{0,n} (x)|^{\beta} \,dx \leq K n^{-\beta } \int_{1}^{\infty} x^{(\alpha -1)\beta}\, dx \leq K n^{-\beta} 
\end{align} 
(the same estimate holds for the function $g_{0,n}$) and
\begin{align} \label{est7}
B_n=\left| \int_{-\frac{2}{n}}^{0} |g_{0,n} (x)|^{\beta} \,dx - \int_{-\frac{2}{n}}^{0} | h_{0,n} (x)|^{\beta} \,dx \right|
\leq K n^{-\alpha \beta -2}.
\end{align}
Finally, applying the methods of \eqref{est4}--\eqref{est3}, we deduce the inequality 
\begin{align} \label{est8}
&\left| \int_{- 1 }^{-\frac 2n} |g_{0,n} (x)|^{\beta} \,dx - 
\int_{- 1 }^{-\frac 2n} | h_{0,n} (x)|^{\beta}\, dx \right| \\
 &\qquad\leq K \Big( \int_{- 1 }^{-\frac 2n} | h_{0,n} (x)|^{\beta} x \,dx + n^{-1} 
\int_{- 1 }^{-\frac 2n} | h_{0,n} (x)|^{\beta-1} x^{\alpha } \,dx \Big) \nonumber \\
&\qquad \leq K 
n^{-\beta}. \nonumber
\end{align} 
Now, note that $\alpha \beta +2>\beta $ since $\alpha >0$ and $\beta \in (1,2)$. We obtain the assertion of part (i) by \eqref{est6} and \eqref{est7}.
\end{proof}
Now, we introduce an auxiliary continuous function $u=u_n: \R_{+} \mapsto \R$, which satisfies the inequality
\begin{align} \label{ufunction}
|u(x )| \leq{}&  K \left(|x|^{\alpha} \1_{[0,k+1]} (x)+ |x|^{\alpha-k} \1_{[k+1, n  )} (x) \right.\\
{}&\left. \phantom{ K \Big(} +  
n^{\alpha-k }(\1_{[n , n+k]}(x) +
v((x-k)/n) \1_{(n+k ,\infty )}(x)) \right),
\end{align}
where  $v\in L^{\beta }((1 ,\infty ))
\cap C((1 ,\infty ))$ is a decreasing function. 
The next lemma
presents some approximations for certain integrals of $u$. 

\begin{lem} \label{helpLemma}
Suppose that $u$ is a function satisfying \eqref{ufunction},   $q\in [0,\beta/2)$, $k\geq 2$, $\alpha <k-2/\beta $ and set 
\begin{align*}
I_{l,n,q}:= \int_{0}^{\infty}    |u (x+l)| ^{\beta-q} |u (x)|^{q}\,dx. 
\end{align*}
\begin{enumerate}
\item[(i)] \label{item-help-1}There exists   $r=r_{\alpha,\beta,q,k}>1$ such that  for all $n\geq 1,\, l=0,\dots,n$ we have   
\begin{align} \label{ilnqest}
I_{l,n,q} \leq K (l^{-r} + n^{-r}).
\end{align}
\item[(ii)] There exists $r=r_{\alpha,\beta,q,k}>0$ such that for all $m\geq 0$ and all  functions $u$ satisfying \eqref{ufunction} and vanishing on $[0,m]$  we have 
\begin{align} \label{eq:Ilnb}
\limsup_{n\to\infty}\sum_{l=0}^{n-1} I_{l,n,q}
 \leq K m^{-r}.
\end{align}
\end{enumerate}
\end{lem}

\begin{proof}
(i) To show the estimate \eqref{ilnqest} we decompose $I_{l,n,q}$ into four terms as follows. 
First applying the estimate at \eqref{ufunction} we deduce that
\[
\int_0^{k+1}    |u (x+l)|^{\beta-q} |u (x)|^{q} \,dx \leq K l^{(\alpha -k)(\beta -q)}. 
\] 
On the other hand, we obtain that 
\begin{align*}
&\int_{k+1}^{n }    |u (x+l)| ^{\beta-q}  |u (x)|^{q} \1_{\{x+l<n \}} \,dx  \\
& \qquad \leq K l^{(\alpha -k)\beta} \int_{k+1}^{n } \left( \frac xl + 1 \right)^{(\alpha -k)(\beta -q)}
\left( \frac xl  \right)^{(\alpha -k)q} dx \\
&\qquad  \leq K l^{(\alpha -k)\beta+1} \int_{(k+1)/l}^{\infty} (y+1)^{(\alpha -k)(\beta -q)} y^{(\alpha -k)q}  \,dy \\ 
&\qquad \leq K \max\{\log(l)l^{(\alpha -k)\beta + 1}, l^{(\alpha -k)(\beta -q)}\}. 
\end{align*} 
For the last approximation we used the fact that $(\alpha -k)\beta<-1$, which insures the integrability at infinity, while for the 
integrability near $0$ we distinguished the cases $(\alpha -k)q >-1$ and $(\alpha -k)q \leq -1$. Observing that the function $v$ introduced at \eqref{ufunction} is continuous and hence bounded on compact sets, we conclude that 
\begin{align*}
&\int_{k+1}^{n }    |u (x+l)| ^{\beta-q}|u (x)|^{q} \1_{\{x+l>n \}} \,dx  
 \leq K n^{(\alpha -k)(\beta-q)} 
\int_{k}^{n } x^{(\alpha-k)q} \,dx \\
&\qquad  \leq K
\begin{cases}
 \log(n)n^{(\alpha-k)\beta +1} & (\alpha-k)q \geq -1 \\
  n^{(\alpha -k)(\beta-q)}  & (\alpha-k)q < -1.
\end{cases} 
\end{align*} 
Finally, for the last term we have that
\[
\int_{n}^{\infty}    |u (x+l) |^{\beta-q}|u (x)|^{q} \,dx \leq K n^{(\alpha -k)\beta } \left(1+
\int_{n}^{\infty} v(x/n)^{\beta} \,dx \right) \leq K  n^{(\alpha -k)\beta +1},
\]
where we used that the function $v$ is decreasing on $(1 ,\infty )$ and $v\in L^{\beta }((1 ,\infty ))$. By noticing that $\max\{(\alpha -k)\beta + 1, (\alpha -k)(\beta -q)\}<-1$,  we obtain  \eqref{ilnqest} by a combination of  the above four estimates. 

\noindent
(ii) To show \eqref{eq:Ilnb} suppose that the function $u$, in addition,  vanish on $[0,m]$. By the decomposition and estimates use in (i) above we have 
\begin{equation}
\limsup_{n\to\infty}\sum_{l=0}^{n-1} I_{l,n,q}\leq \limsup_{n\to\infty}\sum_{l=0}^n \tilde I_{l,n,q}  
\end{equation}
where 
\begin{equation}
\tilde I_{l,n,q}   := \int_{m}^{n }    |u (x+l)| ^{\beta-q}  |u (x)|^{q} \1_{\{x+l<n \}} \,dx.
\end{equation}
For all $l\leq m$,
\begin{equation}
\tilde I_{l,n,q}^m   \leq K l^{(\alpha -k)\beta+1} \int_{m/l}^{\infty} (y+1)^{(\alpha -k)(\beta -q)} y^{(\alpha -k)q}  \,dy 
\leq K m^{(\alpha -k)\beta + 1}
\end{equation}
and hence 
\begin{equation}
 \limsup_{n\to\infty}\sum_{l=0}^n \tilde I_{l,n,q} \leq K\Big( m^{(\alpha -k)\beta + 1} + \sum_{l=m+1}^\infty \max\{\log(l)l^{(\alpha -k)\beta + 1}, l^{(\alpha -k)(\beta -q)}\} \Big) , 
\end{equation}
showing \eqref{eq:Ilnb}. 
\end{proof}
Now, we present one of the main results, which provides estimates for various covariance functions. 
Assume that a given function $u=u_{n}$ satisfies \eqref{ufunction} and 
define
\begin{align}\label{eq:83736}
\theta _{n}(l)={}&   a_p^{-2} \int_{\R^2} \frac{1}{|s_1s_2|^{1+p}} \psi_l (s_1,s_2) \,ds_1 ds_2, \qquad \text{where}\\
\psi_l (s_1,s_2) ={}&  \exp \left(- \int_{\R} |s_1u(x) -s_2u (x+l)|^{\beta } \,dx \right) \nonumber  \\[1.5 ex]
{}&- \exp \left(-  \int_{\R} |s_1 u(x)|^{\beta} + |s_2u(x+l)|^{\beta} \,dx \right). \
\end{align}

\begin{lem} \label{helplem2}
Suppose that $u$ is a function satisfying \eqref{ufunction} and  $\theta_n(l)$ is defined in \eqref{eq:83736}.
\begin{itemize}
\item[(i)] Assume that $k=1$ and $\alpha <1-1/\beta $. Then, for any $l\geq 1$, we obtain
\begin{align} \label{thetalineq1}
|\theta_n (l)|\leq K l^{(\alpha -1)\beta +1}.
\end{align}
\item [(ii)] Suppose that $k\geq 2$ and $\alpha <k-2/\beta $. There exists   $r=r_{\alpha,\beta,q,k}>1$ such that  for all $n\geq 1,\, l=0,\dots,n$ we have   
\begin{align} \label{ilnqest-2}
\theta_n(l)\leq K (l^{-r} + n^{-r}).
\end{align}
\item [(iii)] Suppose that $k\geq 2$ and  $\alpha <k-2/\beta $. 
There exists $r=r_{\alpha,\beta,q,k}>0$ such that for all $m\geq 0$ and all  functions $u$ satisfying \eqref{ufunction} and vanishing on $[0,m]$  
\begin{align} 
\limsup_{n\to\infty}\sum_{l=0}^{n-1} \theta_n(l)\leq  K m^{-r}.
\end{align}
%
%
\end{itemize}
\end{lem}

\begin{proof} 
(a) We decompose the first integral via $\int_{\R^2} = \int_{[-1,1]^2} + \int_{\R^2 \setminus[-1,1]^2}$. In this part we will 
analyze the second integral.     
We need the following identity to deal with large $s_{1},s_{2}$: 
\begin{align*}
&\psi _{l}(s_{1},s_{2})= \exp \left(- ( | s_{1} | ^{\beta }+ | s_{2} | ^{\beta })\int_{\mathbb{R}} | u(x) | ^{\beta } dx\right)\\
& \times
\sum_{r\geq 1}\frac{  1}{r!}\left( \int_{\mathbb{R}}\left| s_{1}u(x)-s_{2}u(x+l) \right|^{\beta } dx -\int_{\mathbb{R}} | s_{1}u(x) | ^{\beta }+ | s_{2}u(x+l) | ^{\beta }\,dx\right)^{r}.
\end{align*}
We will make much use of the inequality    
\begin{align} \label{mainest}
 \left| |x -y|^\beta  - |x|^\beta  -|y|^\beta  \right|
\leq 
\begin{cases} K \left(\min\{|x|,|y|\}^{\beta} + \min\{|x|,|y|\} \max\{|x|,|y|\}^{\beta -1} \right) & \beta >1 \\ 
K\min\{|x|,|y|\}^{\beta}  & \beta \leq 1.
\end{cases}
\end{align}
We start with the case $\beta \leq 1$. According to Lemma \ref{helpLemma}, for each $K>0$ we can find $l_0, n_0\in \N$ such that 
for all $l\geq l_0$, $n\geq n_0$ we have that $K':=I_{l,n,0}<K$. For such $l$ and $n$ we have
\begin{align}
\label{theta1} \theta (l)^{{ \beta \leq 1}}_1:={}& \int_{\mathbb{R}^{2}\setminus [-1,1]^{2}}\frac{|\psi _{l}(s_{1},s_{2})|}{ | s_{1}s_{2} | ^{1+p}}\,ds_{1}\,ds_{2}\\
\leq{}&  K \int_{\R^2 \setminus[-1,1]^2}  \frac{1}{|s_1s_2|^{1+p}} \exp(-K  (|s_1|^{\beta}+ |s_2|^{\beta})) \sum_{r\geq 1}\frac{ 1}{r!} 
| s_{2} | ^{\beta r} 
 I_{l,n,0}^{r}\,ds_{1}\,ds_{2}  \\
\leq {}& K I_{l,n,0} \int_{\R^2 \setminus[-1,1]^2}  \frac{1}{|s_1s_2|^{1+p}} \exp(-K |s_1|^{\beta}+ (K'-K) |s_2|^{\beta}) \,ds_1\,ds_2,
\end{align}
where the latter integral is finite since $K'<K$. For $l<l_0$, we trivially have that 
\begin{align} \label{theta1.1}
|\theta _{n}(l)|<K,
\end{align}
where we use that $|\psi _{l}(s_{1},s_{2})|<2$ by definition. 
Let us now consider the case $\beta >1$. Applying exactly the same arguments as for $\beta \leq 1$, 
we deduce for $l$ large enough 
\begin{align} \label{theta1.2}
\theta (l)_{1}^{\beta >1}\leq {}&  K\int_{\mathbb{R}^{2}\setminus [-1,1]^{2}}\frac{\left( -K( | s_{1} | ^{\beta }+ | s_{2} | ^{\beta })\right)}
{ | s_{1}s_{2} | ^{1+p}} \\
{}& \times \sum_{r\geq 1}\frac{ 1}{r!}\left(  | s_{1} | ^{\beta -1} | s_{2} | I_{l,n,\beta -1}+ | s_{2} | ^{\beta }I_{l,n,0}\right)^{r} ds_{1}\,ds_{2}\\
\leq{}&  K( I_{l,n,\beta -1 }+ I_{l,n,0}).
\end{align}
This estimate completes the first part of the proof. 
\\ \\
(b)  Given the boundedness of the exponential function on compact intervals and symmetry arguments, we are left with discussing 
the quantity 
\begin{align} \label{theta2}
 \theta (l)_2:= \int_{[0,1]^2}  \frac{1}{|s_1s_2|^{1+p}} 
\Big(  \int_{\R} \Big| | s_1 { u }(x)-s_2{ u }(x+l)|^{\beta} 
 - |s_1 { u }(x)|^{\beta} - |s_2{ u }(x+l)|^{\beta} \Big |\, dx \Big)\,ds_1\,ds_2.
\end{align}
This requires a much deeper analysis. Clearly, we obtain the inequality
\[
\theta (l)_2 \leq K(\theta (l)_{2.1} + \theta (l)_{2.2})
\]
with 
\begin{align}
\theta (l)_{2.1} = {}&  \int_{[0,1]^2}  \frac{1}{|s_1s_2|^{1+p}} 
\Big(  \int_{0}^{\infty} \min\{|s_1 { u }(x)|^{\beta}, |s_2{ u }(x+l)|^{\beta}\} \,dx \Big)\,ds_1\,ds_2 \\
\theta (l)_{2.2} ={}&  \int_{[0,1]^2}  \frac{1}{|s_1s_2|^{1+p}} 
\Big(  \int_{0}^{\infty} \min\{|s_1 { u }(x)|, |s_2{ u }(x+l)| \} \\
&{} \phantom{\int_{[0,1]^2}  \frac{1}{|s_1s_2|^{1+p}} 
\Big(  \int_{0}^{\infty} }  \times \max\{|s_1 { u }(x)|^{\beta-1}, |s_2{ u }(x+l)|^{\beta-1}\} \,dx \Big)\,ds_1\,ds_2
\end{align}
where the  term $\theta (l)_{2.2}$ appears only in case $\beta >1$. We start by handling the quantity $\theta (l)_{2.1}$. Set 
$a_k(x,l):= |{ u }(x+l)|/|{ u }(x)|\in \mathbb{R}\cup \{\infty \}$ for $x>0$, and recall the convention $\int_{\infty }^{a}=0$ for any $a\in \mathbb{R}\cup \{\pm\infty \}$.  Notice that we can rearrange the involved integrals due to positivity of the integrand. Hence,
we deduce
\[
\theta (l)_{2.1} = \theta (l)_{2.1.1} + \theta (l)_{2.1.2}
\]  
with
\begin{align}
\theta(l) _{2.1.1}&=\int_{[0,1]^{2}}(s_{1}s_{2})^{-1-p}\Big(\int_{0}^{\infty } | s_{1}u {(x)}  | ^{\beta } \1_{\{s_{1}\leq s_{2}a_k(x,l)\}}\,dx\Big)\,ds_{1}\,ds_{2}\\
&=\int_{0}^{\infty } |u(x)|^{\beta }\Big(\int_{0}^{1}s_{2}^{-1-p}\int_{0}^{\min\{1,s_{2}a_k(x,l)\}}s_{1}^{-1-p+\beta }\,ds_{1}\,ds_{2}\Big)\,dx. \label{eq:387450}
\end{align}
Next we will estimate the parenthesis in \eqref{eq:387450} as follows 
\begin{align}
{}& \int_{0}^{1}s_{2}^{-1-p}\int_{0}^{\min\{1,s_{2}a_k(x,l)\}}s_{1}^{-1-p+\beta }\,ds_{1}\,ds_{2}\\
\leq {}&  K\int_{0}^{1}s_{2}^{-1-p}\Big(\1_{\{s_{2}a_k(x,l)\leq 1\}}\int_{0}^{s_{2}a_k(x,l)}s_{1}^{-1-p+\beta }\,ds_{1}
 +\1_{\{s_{2}a_k(x,l)\geq 1\}}\int_{0}^{1}s_{1}^{-1-p+\beta }ds_{1} \Big)\,ds_{2}\\
\leq {}& K\int_{0}^{1}s_{2}^{-1-p} \Big(\1_{\{s_{2}a_k(x,l)\leq 1\}}(s_{2}a_k(x,l))^{-p+\beta } 
+\1_{\{s_{2}a_k(x,l)\geq 1 \}} \Big)\,ds_{2}\\
\leq {}& K \Big(\int_{0}^{\min\{1,a_k(x,l)^{-1}\}}a_k(x,l)^{\beta -p}s_{2}^{-1-2p+\beta }ds_{2} 
+\1_{\{a_k(x,l)^{-1}\leq 1\}}\int_{a_k(x,l)^{-1}}^{1}  s_{2}^{-1-p}ds_{2} \Big)\\
\leq {}& K \bigg( a_{k}(x,l)^{\beta -p}\big( \1_{\{a_{k}(x,l)^{-1}\leq 1\}} a_{k}(x,l)^{-(\beta -2p)}\\
{}&\phantom{K \bigg(}+\1_{\{a_{k}(x,l)^{-1}\geq 1\}} \big)+\1_{\{a_{k}(x,l)^{-1}\leq 1\}}(1+a_{k}(x,l)^{p})\Big)\\
\leq {}& K \bigg(\1_{\{|u_{k}(x+l)|\geq u_{k}(x)\}}\left( 2\left| \frac{u_{k}(x+l)}{u_{k}(x)} \right|^{p}+1 \right)+\1_{\{|u_{k}(x+l)|\leq u_{k}(x)\}}\left| \frac{u_{k}(x+l)}{u_{k}(x,l)} \right|^{\beta -p}\bigg).
\end{align}
Hence, 
\begin{equation}
\label{theta2.1.1}
\theta(l) _{2.1.1}\leq K(I_{l,n,0}+I_{l,n,p}).
\end{equation}
%
%
%
%
%
%
We have  
\begin{align}
\theta(l)_{2.1.2}&=\int_{[0,1]^{2}}(s_{1}s_{2})^{-1-p}\Big(\int_{0}^{\infty } | s_{2}u (x+l)   | ^{\beta }\1_{\{s_{2}\leq s_{1}a_k(x,l)^{-1}\}}\,dx\Big)\,ds_{1}\,ds_{2}\\
&=\int_{0}^{\infty } |u (x+l)|^{\beta }\Big(\int_{0}^{1}s_{1}^{-1-p}\int_{0}^{\min\{1,s_{1}a_k(x,l)^{-1}\}}s_{2}^{-1-p+\beta }\,ds_{2}\,ds_{1}\Big)\,dx,\label{eq:6736}
\end{align}
and the parenthesis in \eqref{eq:6736}  is estimated as follows 
\begin{align}
{}& \int_{0}^{1}s_{1}^{-1-p}\int_{0}^{\min\{1,s_{1}a_k(x,l)^{-1}\}}s_{2}^{-1-p+\beta }\,ds_{2}\,ds_{1}\\
&\leq K\Big(\int_{0}^{1}\1_{\{s_{1}\leq a_k(x,l)\}}s_{1}^{-1-p}(s_{1}a_k(x,l)^{-1})^{-p+\beta }\,ds_{1} 
+\int_{0}^{1}\1_{\{s_{1}\geq  a_k(x,l)\}}s_{1}^{-1-p} \,ds_{1}\Big)\\
&\leq K\Big(a_k(x,l)^{p-\beta }\int_{0}^{\min\{1,a_k(x,l)\}}s_{1}^{-1-2p+\beta }\,ds_{1}
 +\1_{\{a_k(x,l)\leq 1\}}\int_{a_k(x,l)}^{1}s_{1}^{-1-p}\,ds_{1}\Big)\\
&\leq K \Big(a_k(x,l)^{p-\beta }\1_{\{a_k(x,l)\geq 1\}}+2a_k(x,l)^{-p}\1_{\{a_k(x,l)\leq 1\}}\Big),
\end{align}
which shows 
\begin{equation}
\label{theta2.1.2} \theta (l)_{2.1.2} \leq  K(I_{l,n,0}+I_{l,n,p}).
\end{equation}
We also have $\theta(l)_{2.2}=\theta(l)_{2.2.1}+\theta(l)_{2.2.2}$ with
\begin{align}
\theta (l)_{2.2.1}&=\int_{[0,1]^{2}}s_{1}^{-1-p}s_{2}^{-1-p}\Big(\int_0^\infty\1_{\{s_{1}\leq s_{2}a_k(x,l)\}}s_{1} |u (x)|s_{2}^{\beta -1}|u (x+l)|^{\beta -1}\,dx\Big)\,ds_{1}\,ds_{2}\\
&\leq  \int_{0}^\infty u(x) u (x+l)^{\beta -1}\Big(\int_{0}^{1}s_{2}^{\beta -2+p}\int_{0}^{s_{2}a_k(x,l)}s_{1}^{-p}\,ds_{1}\,ds_{2}\Big)\,dx\\
&\leq K \int_{0}^\infty |u(x)|  |u (x+l)|^{\beta -1}\Big(\int_{0}^{1}s_{2}^{\beta -2+p} (s_{2}a_k(x,l))^{1-p}\,ds_{2}\Big)\,dx\\
&\leq K\int_{0}^\infty |u(x)| |u (x+l)|^{\beta -1}\Big(a_k(x,l)^{1-p}\int_{0}^{1}s_{2}^{\beta -1}\,ds_{2}\Big)\,dx\\
&\leq K\int_{0}^\infty |u(x)|^{p} |u (x+l)|^{\beta -p}\,dx \leq K I_{l,n,p},   \label{theta2.2.1}
\end{align}
and 
\begin{align}
\theta(l)_{2.2.2}&=\int_{[0,1]^{2}}s_{1}^{-1-p}s_{2}^{-1-p}\int_0^\infty \1_{\{s_{2}\leq s_{1}a_k(x,l)^{-1}\}}s_{1} ^{\beta -1}
|u (x)|^{\beta -1}s_{2} |u (x+l)| \,dx\,ds_{1}\,ds_{2}\\
&=\int_0^\infty	  |u (x)|^{\beta -1}|u (x+l)|\Big[\int_{0}^{1}s_{1}^{-2-p+\beta } \Big(\1_{\{s_{1}a_k(x,l)^{-1}\leq 1\}}\int_{0}^{s_{1}a_k(x,l)^{-1}}s_{2}^{ -p}\,ds_{2} 
\\  &
 \phantom{=\int_0^\infty	  |u (x)|^{\beta -1}|u (x+l)|\int_{0}^{1}s_{1}^{-2-p+\beta \Big(} \Big(}
+\1_{\{s_{1}a_k(x,l)^{-1}\geq  1\}}\int_{0}^{1}s_{2}^{ -p}ds_{2} \Big)\,ds_{1}\Big]\,dx.\quad \label{eq:7364}
\end{align}
The  bracket in \eqref{eq:7364}  is estimated as follows 
\begin{align}
{}& \int_{0}^{1}s_{1}^{-2-p+\beta } \Big(\1_{\{s_{1}a_k(x,l)^{-1}\leq 1\}}\int_{0}^{s_{1}a_k(x,l)^{-1}}s_{2}^{ -p}\,ds_{2} 
+\1_{\{s_{1}a_k(x,l)^{-1}\geq  1\}}\int_{0}^{1}s_{2}^{ -p}ds_{2} \Big)\,ds_{1}
\\ &\leq K \int_{0}^{1}s_{1}^{-2-p+\beta }\1_{\{s_{1}a_k(x,l)^{-1}\leq 1\}}s_{1}^{1-p}a_k(x,l)^{p-1}ds_{1} 
+\1_{\{a_k(x,l)<1\}}\int_{a_k(x,l)}^{1}s_{1}^{-2-p+\beta } ds_{1} \Big)\\
&\leq Ka_k(x,l)^{p-1}\int_{0}^{\min\{1,a_k(x,l)\}}s_{1}^{-1-2p+\beta }  ds_{1}
 +\1_{\{a_k(x,l)\leq 1\}} a_k(x,l)^{-1-p+\beta } ds_{1} \Big)\\
&\leq K\left( a_k(x,l)^{p-1} \1_{\{a_k(x,l)\geq 1\}}+a_k(x,l)^{-1-p+\beta }\1_{\{a_k(x,l)\leq 1\}} \right),
\end{align}
which shows 
\begin{equation}\label{theta2.2.2}
 \theta(l)_{2.2.2} \leq K\left( I_{l,n,0}+I_{l,n,p} \right).
\end{equation}
Combining the estimates \eqref{theta1}, \eqref{theta1.1}, \eqref{theta1.2}, \eqref{theta2.1.1}, \eqref{theta2.1.2}, 
\eqref{theta2.2.1} and \eqref{theta2.2.2}, and applying now 
Lemma~\ref{helpLemma}(i)--(ii) we obtain the desired assertion.
\end{proof}

\section*{Acknowledgements}
We would like to thank Donatas Surgailis, who helped us to understand the limit theory for discrete moving 
averages.

\bibliographystyle{chicago}
 
\end{document}